\documentclass[a4paper, oneside]{amsart}

\usepackage[
paper=a4paper,
text={132mm,205mm},centering
]{geometry}  
\usepackage{amssymb}
\usepackage[all]{xy}
\usepackage{graphicx}
\usepackage{hyperref}
\usepackage{tikz-cd}

\usetikzlibrary{arrows,arrows.meta}

\iftrue
\makeatletter
\def\@settitle{%
  \vspace*{-0pt}
  \begin{flushleft}%
    \LARGE\bfseries
    \strut\@title\strut
  \end{flushleft}%
}

\def\@setabstracta{%
    \ifvoid\abstractbox
  \else
    \skip@17pt \advance\skip@-\lastskip
    \advance\skip@-\baselineskip \vskip\skip@
    \box\abstractbox
    \prevdepth\z@ 
    \vskip-28pt
  \fi
}
\renewenvironment{abstract}{%
  \ifx\maketitle\relax
    \ClassWarning{\@classname}{Abstract should precede
      \protect\maketitle\space in AMS document classes; reported}%
  \fi
  \global\setbox\abstractbox=\vtop \bgroup
    \normalfont\small
    \list{}{\labelwidth\z@
      \leftmargin0pc \rightmargin\leftmargin
      \listparindent\normalparindent \itemindent\z@
      \parsep\z@ \@plus\p@
      
    }%
    \item[\hskip\labelsep\bfseries\abstractname.]%
}{%
  \endlist\egroup
  \ifx\@setabstract\relax \@setabstracta \fi
}

\def\ps@headings{\ps@empty
  \def\@evenhead{%
    \setTrue{runhead}%
    \normalfont\scriptsize
    \rlap{\thepage}\hfill
    \def\thanks{\protect\thanks@warning}%
    \leftmark{}{}}%
  \def\@oddhead{%
    \setTrue{runhead}%
    \normalfont\scriptsize
    \def\thanks{\protect\thanks@warning}%
    \rightmark{}{}\hfill \llap{\thepage}}%
  \let\@mkboth\markboth
}\ps@headings

\def\section{\@startsection{section}{1}%
  \z@{-1.4\linespacing\@plus-.5\linespacing}{.8\linespacing}%
  {\normalfont\bfseries\Large}}
\def\subsection{\@startsection{subsection}{2}%
  \z@{-.8\linespacing\@plus-.3\linespacing}{.5\linespacing\@plus.2\linespacing}%
  {\normalfont\bfseries\large}}
\def\subsubsection{\@startsection{subsubsection}{3}%
  \z@{.7\linespacing\@plus.2\linespacing}{-1.5ex}%
  {\normalfont\bfseries}}
\def\@secnumfont{\bfseries}

\renewcommand\contentsnamefont{\bfseries}
\def\@starttoc#1#2{\begingroup
  \setTrue{#1}%
  \par\removelastskip\vskip\z@skip
  \@startsection{}\@M\z@{\linespacing\@plus\linespacing}%
    {.5\linespacing}{
      \contentsnamefont}{#2}%
  \ifx\contentsname#2%
  \else \addcontentsline{toc}{section}{#2}\fi
  \makeatletter
  \@input{\jobname.#1}%
  \if@filesw
    \@xp\newwrite\csname tf@#1\endcsname
    \immediate\@xp\openout\csname tf@#1\endcsname \jobname.#1\relax
  \fi
  \global\@nobreakfalse \endgroup
  \addvspace{32\p@\@plus14\p@}%
  \let\tableofcontents\relax
}
\def\contentsname{Contents}
\def\l@section{\@tocline{2}{.5ex}{0mm}{5pc}{}}
\def\l@subsection{\@tocline{2}{0pt}{2em}{5pc}{}}
\makeatother
\fi 

\theoremstyle{plain}
\newtheorem{theorem}{Theorem}[section]
\newtheorem{proposition}[theorem]{Proposition}
\newtheorem{corollary}[theorem]{Corollary}
\newtheorem{lemma}[theorem]{Lemma}

\theoremstyle{definition}
\newtheorem{definition}[theorem]{Definition}

\newtheorem{remark}[theorem]{Remark}

\makeatletter
\def\Nopagebreak{\@nobreaktrue\nopagebreak}
\makeatother

\def\Z{\mathbb{Z}}
\def\Q{\mathbb{Q}}

\def\p{\mathfrak{p}}
\def\A{\mathcal{A}}
\def\T{\mathcal{T}}
\def\M{\mathcal{M}}
\def\C{\mathcal{C}}
\def\H{\mathcal{H}}
\def\K{\mathcal{K}}
\def\hM{\widehat{M}}
\def\Aut{\operatorname{Aut}}
\def\Ker{\operatorname{Ker}}
\def\Coker{\operatorname{Coker}}
\def\Im{\operatorname{Im}}
\def\Tor{\operatorname{Tor}}
\def\rank{\operatorname{rank}}
\def\D{\mathsf{D}}
\def\L{\mathsf{L}}
\def\Sp{\operatorname{Sp}}
\def\id{\mathrm{id}}

\def\to{\mathchoice{\longrightarrow}{\rightarrow}{\rightarrow}{\rightarrow}}
\makeatletter
\newcommand{\shortxra}[2][]{\ext@arrow 0359\rightarrowfill@{#1}{#2}}
\def\longrightarrowfill@{\arrowfill@\relbar\relbar\longrightarrow}
\newcommand{\longxra}[2][]{\ext@arrow 0359\longrightarrowfill@{#1}{#2}}
\renewcommand{\xrightarrow}[2][]{\mathchoice{\longxra[#1]{#2}}%
  {\shortxra[#1]{#2}}{\shortxra[#1]{#2}}{\shortxra[#1]{#2}}}
\newcommand{\hooklongrightarrow}{\lhook\joinrel\longrightarrow}

\makeatother

\makeatletter
\def\Nopagebreak{\@nobreaktrue\nopagebreak}
\makeatother

\begin{document}

\title
[Homology cylinders and invariants related to lower central series]
{Homology cobordism group of homology cylinders and invariants related to lower central series}
\author{Minkyoung Song}
\address{Center for Geometry and Physics, Institute for Basic Science (IBS), Pohang 37673, Korea}
\email{mksong@ibs.re.kr}

\begin{abstract}
The homology cobordism group of homology cylinders is a generalization of both the mapping class group of surfaces and the string link concordance group. We consider extensions of Johnson homomorphisms of a mapping class group, Milnor invariants and Orr invariants of links to homology cylinders, all of which are related to free nilpotent groups. We establish a combined filtration via kernels of extended Johnson homomorphisms and extended Milnor invariants. We determine its image under the three invariants, and investigate relations among the invariants, and relations of the filtration to automorphism groups of free nilpotent groups and to graded free Lie algebras. We obtain the number of linearly independent invariants by examining the successive quotients of the filtration.
\end{abstract}

\maketitle

\setcounter{tocdepth}{2}

\section{Introduction}
Let $\Sigma_{g,n}$ be a compact, oriented surface of genus $g$ with $n$ boundary components. Roughly speaking, a \emph{homology cylinder} over $\Sigma_{g,n}$ is a homology cobordism between two copies of~$\Sigma_{g,n}$ endowed with two embeddings $i_+$ and $i_-$ of~$\Sigma_{g,n}$ (see Definition~\ref{def:homology cylinder}). 
Goussarov~\cite{Go} and Habiro~\cite{Habi} independently introduced homology cylinders as important model objects for their theory of finite type invariants of 3-manifolds which play the role of string links in the theory of finite type invariants of links. Garoufalidis and Levine introduced in~\cite{GL, L01} the homology cobordism group $\H_{g,n}$ of homology cylinders over $\Sigma_{g,n}$ (see Definition~\ref{def:homology cobordism} for the description of the group operation). We can regard $\H_{g,n}$ as an enlargement of the mapping class group  $\M_{g,n}$ of $\Sigma_{g,n}$ since there is a natural embedding $\M_{g,n} \to \H_{g,n}$ (see \cite[p.~247]{L01}, \cite[Proposition 2.3]{CFK}).
We can identify $\H_{g,n}$ with more familiar groups. Both $\H_{0,0}$ and $\H_{0,1}$ are isomorphic to the group of homology cobordism classes of integral homology 3-spheres. The group $\H_{0,2}$ is isomorphic to the concordance group of framed knots in homology 3-spheres. For $n\geq 3$, $\H_{0,n}$ is isomorphic to the concordance group of framed $(n-1)$-component string links in homology 3-balls, or equivalently, in homology cylinders over $D^2=\Sigma_{0,1}$. Similarly, $\H_{g,n}$ can be considered to be the concordance group of framed $(n-1)$-component string links in homology cylinders over $\Sigma_{g,1}$.

Hereafter, we assume that $\Sigma_{g,n}$ has nonempty boundary, that is,~$n>0$. 
Hence $F:= \pi_1(\Sigma_{g,n})$ and $H:=H_1(\Sigma_{g,n};\Z)$ are the free group and the free abelian group of rank~$2g+n-1$, respectively. 
We denote by $\D_k(H)$ the kernel of the bracket map $H\otimes \L_k(H)\to \L_{k+1}(H)$ where $\L_k(H)$ is the degree $k$ part of the free Lie algebra over $H$. For a group $G$, $G_k$ denotes the $k$th term of lower central series given by $G_1=G$, $G_{k+1} = [G, G_k]$.

To understand structures of $\H_{g,n}$, several invariants and filtrations are introduced. First, Garoufalidis and Levine extended Johnson homomorphisms and Johnson filtration to $\H_{g,1}$~\cite{GL,L01}. Based on them, Morita, Sakasai, Cochran-Harvey-Horn~\cite{Mo08, Sa, CHH} constructed their new invariants on $\H_{g,1}$, respectively. Cha-Friedl-Kim~\cite{CFK} introduced a torsion invariant on $\H_{g,n}$. The author naturally extended Garoufalidis-Levine's Johnson homomorphisms and the filtration to $\H_{g,n}$ in~\cite{So}. Also, she defined extended Milnor invariants, extended Milnor filtration, and Hirzebruch-type invariants.
In this paper, we newly defined extended Orr invariants, and address the three invariants related to lower central series:

\begin{itemize}
	\item extended Johnson homomorphisms $\tilde\eta_k\colon \H_{g,n}\to \Aut(F/F_k)$ (see Section~\ref{subsec:Johnson})
	\item extended Milnor invariants $\tilde\mu_k\colon \H_{g,n}\to (F/F_k)^{2g+n-1}$ (see Section~\ref{subsec:Milnor})
	\item extended Orr invariants $\tilde\theta_k\colon \H_{g,n}(k)\to H_3(F/F_k)$ (see Section~\ref{subsec:Orr}),
\end{itemize}
with the following filtrations of $\H_{g,n}$:
\begin{itemize}
	\item extended Johnson filtration $\H_{g,n}[k]:=\Ker \tilde\eta_k$ 
	\item extended Milnor filtration $\H_{g,n}(k):=\Ker \tilde\mu_k\cap\H^0_{g,n}(=\Ker\tilde\mu_k\textrm{ for }k\geq 2)$
	\item modified extended Johnson filtration $\H^0_{g,n}[k]:=\H_{g,n}[k]\cap\H^0_{g,n}$.

\end{itemize}
Here, $\H^0_{g,n}$ is the subgroup consisting of 0-framed homology cylinders (see Section~\ref{subsec:0-framed}). By modifying $\H_{g,n}[k]$ to $\H^0_{g,n}[k]$, and we have a filtration combined with $\H_{g,n}(k)$
$$\H\rhd \H^0[1] =\H(1)\rhd \H^0[2] \rhd \H(2)\rhd \cdots \rhd \H^0[k] \rhd \H(k)\rhd \H^0[k+1] \rhd \H(k+1)\rhd \cdots  .$$ 
We determine images and successive quotients of this filtration under the three invariants. 

\subsection{Extended Johnson homomorphims}
\emph{Johnson homomorphisms} $\eta_k\colon\M_{g,1}\to\Aut(F/F_k)$ and the \emph{Johnson filtration} (also called the \emph{relative weight filtration}) on the mapping class group have connections to the geometry of the moduli space of curves, 3-manifold topology, and number theory (e.g., see \cite{Hai95, GL98, Pu12, Ma13}).
However, many basic questions remain open. In particular, the image of the Johnson homomorphisms are unknown. Morita~\cite{Mo93} gave the induced injective homomorphisms $\frac{\M_{g,1}[k]}{\M_{g,1}[k+1]} \hookrightarrow \D_k(H)$, but the precise images are also unknown. Rationally, there are results for $k\leq 7$ in \cite{Jo, Hai, AN, Mo99, MSS}. On the other hand, Garoufalidis and Levine considered \emph{extended Johnson homomorphisms} $\tilde\eta_k\colon\H_{g,1}\to\Aut(F/F_k)$ and the filtration $\H_{g,1}[k]:=\Ker\tilde\eta_k$, and determine $\frac{\H_{g,1}[k]}{\H_{g,1}[k+1]}$ as the whole $\D_k(H)$ in \cite{GL}. The homomorphisms have been used to develop new invariants and to find conditions for the invariants to be (quasi-)additive in \cite{Mo08, Sa, CHH}. The author naturally extended $\tilde\eta_k$ and the filtration to $\H_{g,n}$ in~\cite{So}, and suggested a candidate of the image $\tilde\eta_k(\H_{g,n})$. Now we formulate the images and quotients of the combined filtration.
	
\begin{theorem}
	The image of the extended Johnson homomorphism $\tilde\eta_k\colon\H_{g,n}\to\Aut(F/F_k)$ is 
	\begin{align*}
			\Aut_*(F/F_k)&:=\bigg\{\phi\in \Aut(F/F_k)~\bigg|\begin{array}{ll}\textrm{there is a lift }\tilde\phi \in \Aut(F/F_{k+1}) \textrm{ such that } \\
			  \tilde\phi(x_i)=\alpha_i^{-1} x_i \alpha_i \textrm{ for some }\alpha_i \in F/F_k\textrm{ and }\tilde\phi \textrm{ fixes } [\partial_n]\end{array}\bigg\}
	\end{align*}
	 where $x_1,\ldots, x_{n-1} \in F$ are homotopic to each boundaries.
Moreover, 
	 \begin{align*}
	 \frac{\H_{g,n}[k]}{\H_{g,n}[k+1]} &\cong\frac{\H^0_{g,n}[k]}{\H^0_{g,n}[k+1]} \cong  \Ker\{\Aut_*(F/F_{k+1})\rightarrow\Aut_*(F/F_k)\},\\
	 \frac{\H_{g,n}(k)}{\H^0_{g,n}[k+1]}&\cong \Ker\{\Aut_*(F/F_{k+1})\rightarrow\Aut_*(F/F_k)\} \cap \{\phi\;|\;\phi(x_i)=x_i \}.
	 \end{align*}
\end{theorem}	
We construct a descending filtration of $\Aut_*(F/F_k)$ which corresponds to images of the combined filtration under $\tilde\eta_k$. More details are discussed in Section~\ref{subsec:result_Johnson}.

\subsection{Extended Milnor invariants}
Milnor's $\bar\mu$-invariants of links are actually invariants on string links without indeterminacy~\cite{HL}. Orr~\cite{Or} determined the precise number of linearly independent Milnor invariants of length $k$ as the rank of~$\D_k(H)$. The author defined extended Milnor invariants $\tilde\mu_k\colon\H_{g,n}\to (F/F_k)^{2g+n-1}$ and a filtration $\H_{g,n}(k):=\Ker\tilde\mu_k$ on $\H_{g,n}$ in~\cite{So}. The invariants are not only an extension of the Milnor invariants for (string) links, but also a generalization of extended Johnson homomorphisms of~$\H_{g,1}$. If $n=1$, they are equivalent to the extended Johnson homomorphisms on~$\H_{g,1}$, and if $g=0$, the extended Milnor invariants are equivalent to the Milnor invariants of string links in homology 3-balls. Note that the Orr's result above can be restated that $\frac{\H_{0,n}(k)}{\H_{0,n}(k+1)}\cong \D_k(H)$. In \cite{So}, it is revealed that  for general $g\geq 0$ and $n>0$, the successive quotients $\frac{\H_{g,n}(k)}{\H_{g,n}(k+1)}$ are finitely generated free abelian if $k\geq 2$, and upper and lower bounds of their ranks are provided. Now we complete our previous work as follows. 
 \begin{theorem} \label{theorem:intro}
For $g\geq 0$, $n\geq1$ and $k\geq2$, the extended Milnor invariants induce isomorphisms
	$$\frac{\H_{g,n}(k)}{\H_{g,n}(k+1)} \cong \D_k(H),\quad \frac{\H^0_{g,n}[k+1]}{\H_{g,n}(k+1)} \cong \D_k(H'), \quad\frac{\H^0_{g,n}[2]}{\H_{g,n}(2)}\times\Z^{n-1} \cong \D_1(H')$$
	 where $H'=H_1(\Sigma_{0,n})$.
	Consequently, for $k\geq2$, there are $(2g+n-1)N_k-N_{k+1}$ linearly independent extended Milnor invariants of length $k+1$ distinguishing homology clinders with vanishing extended Milnor invariants of lenght $\leq k$. Here, $N_k=\frac{1}{k}\sum_{d|k} \varphi(d)\; (2g+n-1)^{k/d}$ where $\varphi$ is the M\"obius function.
\end{theorem}
Also, we clarify the image of $\tilde\mu_{k+1}\colon \H_{g,n}(k)\to (F_k/F_{k+1})^{2g+n-1}$ in Theorem~\ref{theorem:Milnor}.
Based on Theorem~\ref{theorem:intro}, the extended Milnor invariants can be regarded as a proper generalization of the extended Johnson homormorphism on $\H_{g,1}$ and the Milnor invariants on $\H_{0,n}$. Habegger showed that the extended Johnson homomorphisms on $\H_{g,1}$ coincides with Milnor's invariants on the string link concordance group, which is called a \emph{Milnor-Johnson correspondence}~\cite{Habe}. It also can be integrated with the extended Milnor invariants.
We remark that vanishing of the extended Milnor invariants is the exact criterion that Hirzebruch-type intersection form defect invariants from iterated $p$-covers are defined~\cite{So}.
 
 \subsection{Extended Orr invariants}
Orr \cite{Or} introduced homotopy invariants of based links to find the number of linearly independent Milnor invariants. Igusa and Orr~\cite{IO} saw the relation of the invariants to $k$-slice and to Milnor invariants to prove the $k$-slice conjecture, which gives a geometric characterization for the vanishing of the Milnor invariants. They reduced the invariants to invariants with values in~$H_3(F/F_k)$. We extend (Igusa-)Orr invariants to homology cylinders in Section~\ref{subsec:Orr}, and obtain the following.

\begin{theorem}
For $k\geq2$, the extended Orr invariant $\tilde\theta_k\colon\H_{g,n}(k) \to H_3(F/F_k)$ is surjective.
It induces isomorphisms
	 \begin{align*}
	  \frac{\H_{g,n}(k)}{\H_{g,n}(2k-1)} &\cong H_3(F/F_k),\\
	 \frac{\H_{g,n}(k)}{\H_{g,n}(k+1)} &\cong \Coker\{H_3(F/F_{k+1})\rightarrow H_3(F/F_k)\}.
	 \end{align*}
\end{theorem}
Cha and Orr defined transfinite Milnor invariants for closed 3-manifolds in~\cite{CO}.  
Our $\tilde\theta_k(M)$ is the same as the Cha-Orr invariant $\theta_k(\hM)$ of finite length in~\cite{CO}, with the closure $\hM$ of~$M$.
	
The above theorems imply relations of the filtration of $\H_{g,n}$ to automorphism groups of free nilpotent groups, graded free Lie algebras, and the third homology of free nilpotent groups, respectively. 
We also establish relations among the three invariants in Section~\ref{subsec:relation}.
Further, we look into $\frac{\H_{g,n}(1)}{\H_{g,n}(2)}$, which is the only non-abelian case in Section~\ref{subsec:first level}. 

\subsection{Organization}
In Section~\ref{sec:def}, we recall definitions of homology cylinders, their homology cobordism group, extended Johnson homomorphisms, and extended Milnor invariants. Also we provide previously known results about the two invariants. We define extended Orr invariants for homology cylinders. In Section~\ref{sec:refine}, we define 0-framed homology cylinders and refine the two filtrations to be totally ordered. In Section~\ref{sec:image}, we determine their images under the three invariants, and build a connection between the invariants. We find the rank of the successive quotients of the filtration. Finally in Section~\ref{sec:proofs}, we give proofs of the theorems stated in Section~\ref{sec:image}.

In this paper, manifolds are assumed to be compact and oriented. Our results hold in both topological and smooth categories.

\subsection*{Acknowledgements}
The author thanks Jae Choon Cha for helpful discussions. This work was supported by IBS-R003-D1.

\section{Definitions and known results} \label{sec:def}
\subsection{Homology cylinder cobordism group}
We recall precise definitions about homology cylinders. Let $\Sigma_{g,n}$, or simply $\Sigma$, be a surface of $g$ genus with $n$ boundary components. We only consider the case~$n>0$.

\begin{definition} \label{def:homology cylinder}
  A \emph{homology cylinder over} $\Sigma$ consists of a 3-manifold $M$ with two embeddings $i_+^{\vphantom{}},~i_-^{\vphantom{}}\colon \Sigma \hookrightarrow \partial M$, called \emph{markings}, such that
  \begin{enumerate}
    \item   $i_+^{\vphantom{}}|_{\partial \Sigma} = i_-^{\vphantom{}}|_{\partial \Sigma}$,
    \item   $i_+\cup i_- \colon \Sigma\cup_\partial (-\Sigma) \to \partial M $ is an orientation-preserving homeomorphism, and
    \item   $i_+^{\vphantom{}}, i_-^{\vphantom{}}$ induce isomorphisms $H_k(\Sigma;\Z)\to H_k(M;\Z)$ for all $k \geq 0$.
  \end{enumerate}
We denote a homology cylinder by $(M,i_+^{\vphantom{}},i_-^{\vphantom{}})$ or simply by~$M$.
\end{definition}

Two homology cylinders $(M,i_+^{\vphantom{}},i_-^{\vphantom{}})$ and $(N,j_+^{\vphantom{}},j_-^{\vphantom{}})$ over $\Sigma_{g,n}$ are said to be \emph{isomorphic} if there exists an orientation-preserving homeomorphism $f\colon M \to N$ satisfying $j_+^{\vphantom{}}=f \circ i_+^{\vphantom{}}$ and $j_-^{\vphantom{}}=f \circ i_-^{\vphantom{}}$. Denote by $\C_{g,n}$ the set of all isomorphism classes of homology cylinders over~$\Sigma_{g,n}$. We define a product operation on $\C_{g,n}$ by
$$(M,i_+^{\vphantom{}},i_-^{\vphantom{}})\cdot (N,j_+^{\vphantom{}},j_-^{\vphantom{}}):=(M\cup_{i_-^{\vphantom{}}\circ (j_+^{\vphantom{}})^{-1}} N, i_+^{\vphantom{}}, j_-^{\vphantom{}})$$
for $(M,i_+^{\vphantom{}},i_-^{\vphantom{}}),~(N,j_+^{\vphantom{}},j_-^{\vphantom{}}) \in \C_{g,n}$, which endows $\C_{g,n}$ with a monoid structure. The identity is $((\Sigma_{g,n} \times I)/(z,0)=(z,t), \id \times 1, \id\times 0)$ where $z\in\partial\Sigma, t\in I$.
Here and after, $I$ denote the interval~$[0,1]$. 

\begin{definition} \label{def:homology cobordism}
  Two homology cylinders $(M,i_+^{\vphantom{}},i_-^{\vphantom{}})$ and $(N,j_+^{\vphantom{}},j_-^{\vphantom{}})$ over $\Sigma_{g,n}$ are said to be \emph{homology cobordant} if there exists a 4-manifold $W$ such that
  \begin{enumerate}
    \item   $\partial W = M \cup (-N) /\sim$, where $\sim$ identifies $i_+^{\vphantom{}}(x)$ with $j_+^{\vphantom{}}(x)$ and $i_-^{\vphantom{}}(x)$ with $j_-^{\vphantom{}}(x)$ for all $x\in\Sigma_{g,n}$, and
    \item   the inclusions $M \hookrightarrow W$, $N \hookrightarrow W$ induce isomorphisms on the integral homology.
  \end{enumerate}
\end{definition}

We denote by $\H_{g,n}$ the set of homology cobordism classes of elements of~$\C_{g,n}$. We usually omit the subscripts and simply write $\H$ for general~$g,n$. By abuse of notation, we also write $(M,i_+,i_-)$, or simply $M$ for the class of~$(M,i_+,i_-)$. The monoid structure on $\C_{g,n}$ descends to a group structure on~$\H_{g,n}$, with $(M,i_+^{\vphantom{}},i_-^{\vphantom{}})^{-1}=(-M,i_-^{\vphantom{}},i_+)$. We call this group the \emph{homology cobordism group} of homology cylinders, or simply \emph{homology cylinder cobordism group}. In fact, there are two kinds of groups $\H_{g,n}^{\mathrm{smooth}}$ and $\H_{g,n}^{\mathrm{top}}$ depending on whether the homology cobordism is smooth or topological, and there exists a canonical epimorphism $\H_{g,n}^{\mathrm{smooth}} \twoheadrightarrow \H_{g,n}^{\mathrm{top}}$ whose kernel contains an abelian group of infinite rank~\cite{CFK}. In this paper, however, the author does not distinguish the two cases since everything holds in both cases. The fact that the mapping class group over $\Sigma_{g,n}$ is a subgroup of $\H_{g,n}$ implies $\H_{g,n}$ is non-abelian except $(g,n)=(0,0), (0,1)$ and~$(0,2)$. For any pairs $(g,n)$ and $(g',n')$ satisfying $n,n'>0$, $g\leq g'$, and $g+n\leq g'+n'$, there is an injective homomorphism $\H_{g,n} \hookrightarrow \H_{g',n'}$~\cite{So}.

\begin{figure}[h]
  \begin{center}
    \includegraphics[scale=.9]{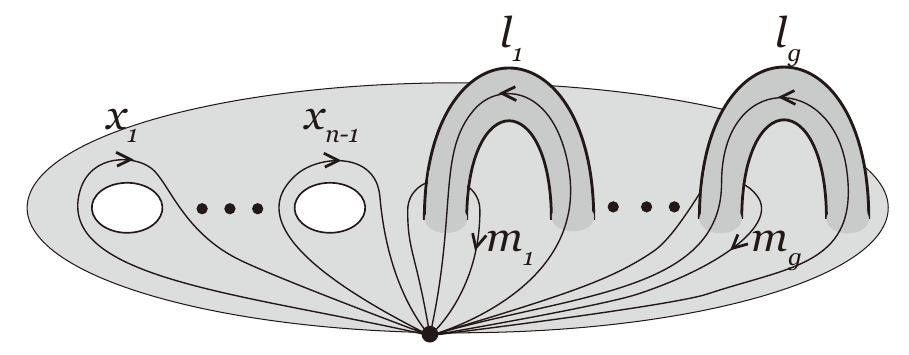}
    \caption{A generating set for $\pi_1(\Sigma_{g,n})$}
    \label{figure:Sigma}
  \end{center}
\end{figure}

Let $\partial_1,\partial_2,\ldots,\partial_n$ be the boundary components of~$\Sigma$. Choose a basepoint $*$ of $\Sigma$ on $\partial_n$ and fix a generating set $\{x_1,\ldots,x_{n-1},m_1,\ldots,m_g,l_1,\ldots,l_g\}$ for $\pi_1(\Sigma,*)$ such that $x_i$ is homotopic to the $i$th boundary component $\partial_i$ and $m_j$, $l_j$ correspond to a meridian and a longitude of the $j$th handle. Since our $n$ is nonzero, the group is free on the above $2g+n-1$ generators. Let $F=\pi_1(\Sigma,*)$ and $H=H_1(\Sigma)$. We may assume that the element $[\partial_n]\in F$ is represented by $\prod_i x_i\prod_j[m_j,l_j]$, see Figure~\ref{figure:Sigma}.

\subsection{Extended Johnson homomorphisms and filtration} \label{subsec:Johnson}
In \cite{GL, L01}, Garoufalidis and Levine defined homomorphisms $\tilde\eta_k\colon \H_{g,1} \to \Aut(F/F_k)$ and a filtration $\H_{g,1}[k]:=\Ker \tilde\eta_k$ as extensions of the Johnson homomorphisms and the Johnson filtration of the mapping class group $\M_{g,1}$~\cite{Jo}. It is a straightforward consequence of Stallings' theorem in~\cite{Sta}. The $\tilde\eta_k$ is defined to be $(i_+)_{k}^{-1}\circ (i_-)_k$ where $(i_\pm)_k \colon F/F_k \to \pi_1(M)/\pi_1(M)_k$ is the isomorphisms induced from~$i_\pm$. They can be considered on $\H_{g,n}$ in the same way. We call the maps $\tilde\eta_k$ on $\H_{g,n}$ \emph{extended Johnson homomorphisms} (they were referred to as `Garoufalidis-Levine homomorphisms'  in \cite{So}). 
They proved that $\tilde\eta_k \colon \H_{g,1} \to \Aut_0(F/F_k)$ are surjective where $F=\pi_1(\Sigma_{g,1})=\langle m_1,\ldots,m_g,l_1,\ldots,l_g\rangle$ and 
    \begin{align*}
  \Aut_0(F/F_k):=\{ \phi \in \Aut(F/F_k)~|~&\textrm{there is a lift }\tilde\phi\colon F/F_{k+1}\to F/F_{k+1} \\ &\textrm{such that }\tilde\phi([\partial_n])=[\partial_n]  \} .
\end{align*}
and the image of $\H_{g,1}[k]$ under~$\tilde\eta_{k+1}$ is isomorphic to $\Ker\{\Aut_0(F/F_{k+1})\to\Aut_0(F/F_{k})\}$. Recall that $\L_k(H)$ is the degree $k$ part of the free Lie algebra over $H$, and $\D_k(H):=\Ker\{H\otimes \L_k(H)\to \L_{k+1}(H)\}$. For $k\geq 2$, $$\D_k(H)\cong\Ker\{\Aut_0(F/F_{k+1})\rightarrow\Aut_0(F/F_{k})\} .$$ Since $\D_k(H)$ is free abelian and the rank is well-known, the successive quotient groups $\frac{\H_{g,1}[k]}{\H_{g,1}[k+1]}$, which is isomorphic to the image $\tilde\eta_{k+1}(\H_{g,1}[k])$, are clarified. We remark that the image $\eta_{k+1}(\M_{g,1}[k])$ of the Johnson subgroup $\M_{g,1}[k]:=\H_{g,1}[k]\cap \M_{g,1}$ of the mapping class group is unknown.

\begin{remark} From the result about the Artin representation for the concordance group of string links in $D^2\times I$, Habegger and Lin \cite[Theorem~1.1]{HL98} proved that $\tilde\eta_k \colon \H_{0,n} \to \Aut_1(F/F_k)$ is surjective where $F=\pi_1(\Sigma_{0,n})=\langle x_1, \ldots, x_{n-1}\rangle$ and 
    \begin{align*}
  \Aut_1(F/F_k):=\{ \phi \in \Aut(F/F_k)~|~&\phi(x_i)=\lambda_i^{-1} x_i \lambda_i
  \textrm{ for some }\lambda_i \in F/F_{k-1} \\
  &\textrm{ and }\phi(x_1\cdots x_{n-1})=x_1\cdots x_{n-1} \} .
\end{align*}
Note that $$\D_k(H)\cong \Ker\{\Aut_1(F/F_{k+2})\rightarrow \Aut_1(F/F_{k+1})\}$$ for $k\geq 2$, and the Artin representation corresponds to the Milnor invariant on string links. Two invariants give rise to the same filtration whose the graded quotients are isomorphic to~$\D_k(H)$. See \cite{HL98, Or}.\end{remark}

\subsection{Extended Milnor invariants and filtration} \label{subsec:Milnor}
Let $(M, i_+^{\vphantom{}},i_-^{\vphantom{}})$ be a homology cylinder over~$\Sigma_{g,n}$. The chosen $x_i^{\vphantom{}}$ is of the form $[\alpha_i^{-1}\cdot\beta_i\cdot \alpha_i]$ for a closed path $\beta_i \colon I \to I/\partial I\xrightarrow{\simeq} \partial_i$ such that the latter map is a homeomorphism and a path $\alpha_i$ from $\beta_i(0)$ to $*$.
Consider the loop $(i_+\circ\alpha_i^{-1})\cdot(i_-\circ\alpha_i)$ in $M$. If $M$ were a framed string link exterior, this loop would be its $i$th longitude. Recall that $(i_\pm)_{k}$ are the induced isomorphisms $F/F_k \to \pi_1(M)/\pi_1(M)_k$. We define $\mu_k(M)_i := (i_+)_{k}^{-1} ([(i_+\circ\alpha_i^{-1}) \cdot (i_-\circ\alpha_i)])$ for the class of the loop $(i_+\circ\alpha_i^{-1})\cdot(i_-\circ\alpha_i)$ in $\pi_1(M,i_+(*))$. It is independent of the choice of $\alpha_i$, and it depends only on the choice of $x_i^{\vphantom{}}$ in~$\pi_1(\Sigma,*)$.
Also we define $\mu'_k(M)_j:=(i_+)_{k}^{-1}([(i_+\circ m_j^{-1})\cdot (i_-\circ m_j)])$ and  $\mu''_k(M)_j:=(i_+)_{k}^{-1}([(i_+\circ l_j^{-1})\cdot (i_-\circ l_j)])$. We denote the ($2g+n-1$)-tuple of $\mu_k(M)_i$, $\mu_k'(M)_j$, and $\mu''_k(M)$ of~$F/F_k$ as~$\tilde\mu_k(M)$. 

In \cite{So}, the author proved that $\tilde\mu_k$ is a crossed homomorphism from $\H_{g,n}$ to $(F/F_k)^{2g+n-1}$ with
$$\tilde\mu_k(MN)_\bullet=\tilde\mu_k(M)_\bullet\cdot \tilde\eta_k(M)(\tilde\mu_k(N)_\bullet),$$
and is a homomorphism on $\Ker\tilde\mu_{k-1}$ and $\Ker\tilde\eta_k$ for~$k>2$. We denote by $\H_{g,n}(k)$ the kernel of~$\tilde\mu_k$. All the $\H_{g,n}(k)$ are normal subgroups of $\H_{g,n}$, and $\{\H_{g,n}(k)\}_k$ is a descending filtration. Since there is an injection $\frac{\H(k)}{\H(k+1)} \hookrightarrow \D_k(H)$ for $k\geq 2$, it turned out that the successive quotient groups $\frac{\H(k)}{\H(k+1)}$ are finitely generated free abelian if~$k\geq 2$. 

There is a relation between extended Johnson homomorphisms and extended Milnor invariants that 	
\begin{equation} \label{eqn:relation}
	\begin{split}
		(\tilde\eta_k(M))(x_i)&= \mu_k(M)_i^{-1}\, x_i \,\mu_k(M)_i \\
		(\tilde\eta_k(M))(m_j)&= m_j \, \mu'_k(M)_j \\
		(\tilde\eta_k(M))(l_j)&= l_j \, \mu''_k(M)_j\qquad \textrm{ in } F/F_k.
	\end{split}
\end{equation}
Remark that 
the definitions of $\mu'_k$ and $\mu''_k$ are slightly different from those in \cite{So}.	
	
\subsection{Extended Orr invariants} \label{subsec:Orr}
To find the number of linearly independent Milnor invariants of a length, Orr defined concordance invariants $\theta_k$ of based links in $S^3$ with vanishing length $\leq k$ Milnor invariants, whose value is in $\pi_3(K_k)$ where $K_k$ is a mapping cone of the inclusion $K(F,1)\to K(F/F_k,1)$ in~\cite{Or}. Here, $K(G,1)$ is an Eilenberg-MacLane space,
Let $L$ be a link in $S^3$ for which Milnor invariants of length $\leq k$ vanish, and $\tau$ be a \emph{basing} $F\to \pi_1(E_L)$ for the exterior $E_L$ of~$L$. Then, 0-framed longitudes of $L$ are in $\pi_1(E_L)_k$, so $F/F_k\cong \pi_1(E_L)/\pi_1(E_L)_k$. We have a map $E_L\to K(\pi_1(E_L),1)\to K(F/F_k,1) \to K_k$ which sends meridians to null-homotopic loops, and so the map extends to $S^3\to K_k$. The homotopy class is defined as $\theta_k(L,\tau) \in \pi_3(K_k)$.
He found the precise number by verifying that the Milnor invariants of length $k+1$ vanish if and only if $\theta_k$ vanishes in $\Coker\{\pi_3(K_{k+1})\to\pi_3(K_k)\}$, which is isomorphic to $\Coker\{H_3(F/F_{k+1})\to H_3(F/F_k)\}$. Igusa and Orr considered invariants valued in $H_3(F/F_k)$ by composing the Hurewicz homomororphism with $\theta_k$, and denoted them by $\bar\theta_k$ in~\cite{IO}. This invariant vanishes for a based link if and only if the Milnor invariants of length $\leq 2k-1$ vanish for that link. 
The invariant $\bar\theta_k(L,\tau)\in H_3(F/F_k)$ is equal to the image of the fundamental class $[M_L]\in H_3(M_L)$ of the zero-surgery manifold $M_L$ of $L$ under a map $M_L\to K(F/F_k,1)$ induced from the basing.

We can define an analogous invariant $\tilde\theta_k\colon\H(k) \to H_3(F/F_k)$. 
Recall that for a homology cylinder $M$, there is an associated closed manifold $\hM$ obtained from $M$ by identifying $i_+(z)$ and $i_i(z)$ for each $z\in\Sigma$, which is called the \emph{closure} of $M$. For a homology cylinder $M$ with trivial $\tilde\mu_k$, we have an isomorphism $F/F_k\cong \pi_1(\hM)/\pi_1(\hM)_k$, see \cite{So}. Define $\tilde\theta_k(M)\in H_3(F/F_k)$ as the image of of $[\hM]$ under $\hM\to K(F/F_k,1)$.
It is a homology cobordism invariant, and is a homomorphism.
We note that $\tilde\theta_k(M)$ is the same as the Cha-Orr invariant $\theta_k(\hM)\in H_3(F/F_k)$ in~\cite{CO}. It can be also thought as a generalization of the (Igusa-)Orr invariant $\theta_k(L,\tau)\in H_3(F/F_k)$ to based links in a homology 3-sphere corresponding to~$M$ along the map from $\H$ to the set of based links in homology 3-spheres. The map will be appear in a future paper.

\section{Refinement and combination of the two filtrations} \label{sec:refine}
We consider two descending filtrations of $\H_{g,n}$: extended Johnson filtration $\H_{g,n}[k]=\Ker \tilde\eta_k$ and extended Milnor filtration $\H_{g,n}(k)=\Ker\tilde\mu_k$
$$\H=\H[1]\,\rhd \H[2]\, \rhd \H[3]\, \rhd \cdots \rhd \,\H[k-1]\, \rhd \H[k]\,\rhd \H[k+1]\, \rhd \cdots $$
$$\H=\H(1)\rhd \H(2) \rhd \H(3) \rhd \cdots \rhd \H(k-1) \rhd \H(k)\rhd \H(k+1) \rhd \cdots $$
Though $\H[k]\supseteq\H(k)$, the two filtrations are not linearly ordered since $\H(k)\nsupseteq\H[k+1]$ unless $n= 1$. Only if $n=1$, $\H_{g,1}[k]=\H_{g,1}(k)$. We will refine and combine them to derive a linearly ordered filtration in Section~\ref{subsection:refine}.

As mentioned in Section~2.2, it is revealed that in the case of $n=1$,
$\frac{\H_{g,1}[k]}{\H_{g,1}[k+1]}$ is isomorphic to $\Ker\{\Aut_0(F/F_{k+1})\to\Aut_0(F/F_{k})\}\cong \D_{k}(H)$ for $k\geq 2$, and in the case of $g=0$, 
$\frac{\H_{0,n}(k)}{\H_{0,n}(k+1)}$ is isomorphic to $\Ker\{\Aut_1(F/F_{k+2})\to \Aut_1(F/F_{k+1})\}\cong\D_k(H)$ for $k\geq 2$.
We will merge and extend the theories on $\H_{g,1}$ and $\H_{0,n}$ to $\H_{g,n}$ by using $\H_{g,n}(k)$, that is, we will show that $\frac{\H_{g,n}(k)}{\H_{g,n}(k+1)}$ is isomorphic to $\D_k(H)$, but $\frac{\H_{g,n}[k]}{\H_{g,n}[k+1]}$ is not. 
We remind that $F$, $H$ and $\D_k(H)$ depend on $g$ and~$n$ since $F=\pi_1(\Sigma_{g,n})$, $H=H_1(\Sigma_{g,n})$.

\begin{remark}
The weight filtration of $\pi_1(\Sigma_{g,n})$ and the induced filtration of the mapping class group are studied in mixed Hodge theory
. They are the same as the lower central series and the Johnson filtration only when $n\leq 1$. But, the graded quotients are also not isomorphic with those in this paper for $n> 1$. 
\end{remark}

\subsection{0-framed homology cylinders} \label{subsec:0-framed}

We consider a natural surjection $H_1(F) \to H_1(F/\langle\langle x_{i'},m_j,l_j~|~i' \neq i\rangle\rangle)\cong \Z$ for each $i=1,\ldots, n-1$. Denote the image of $\mu_2(M)_i$ by $\bar\mu_2(M)_i$. When a homology cylinder $M$ over $\Sigma_{g,n}$ is thought as a framed string link in a homology-($\Sigma_{g,1}\times I$), $\bar\mu_2(M)_i$ indicates ``self-linking number'' of the $i$-th strand. We say that a homology cylinder $M$ is \emph{$0$-framed} if $\bar\mu_2(M)_i=0$ for each $i$. We denote by $\H^0_{g,n}$ the subgroup of $\H_{g,n}$ consisting of the 0-framed homology cylinders.

\begin{proposition}
	There is an exact sequence which splits:
	$$1\to \H^0_{g,n} \to \H_{g,n} \to \Z^{n-1} \to 1.$$
	Therefore, $\H_{g,n}\cong \H_{g,n}^0 \times \Z^{n-1}$.
\end{proposition}

If we consider $\M_{g,n}$ as a subgoup of $\H_{g,n}$, $\varphi\in\M_{g,n}$ corresponds to $I_\varphi:=(\Sigma\times I, \id, \varphi)\in\H_{g,n}$; see \cite[page~247]{L01}, \cite[Proposition~2.4]{CFK}. Note that $I_\varphi=(\Sigma\times I, \varphi^{-1},\id)$ and $(I_\varphi)^{-1}=(\Sigma\times I, \varphi,\id)=(\Sigma\times I, \id, \varphi^{-1})=I_{\varphi^{-1}}$ in $\H$.
For $i=1, \ldots, n-1$, 
we denote by $\tau_i$ the element of the mapping class group $\M_{g,n}$ corresponding to a Dehn twist along the boundary~$\partial_i$.
Let $\T$ be the subgroup of $\M$, and also of $\H$, generated by $\tau_1,\ldots,\tau_{n-1}$.
In general, $I_\varphi$ and a homology cylinder $M$ do not commute, but $I_\tau M = M I_\tau$ for $I_\tau\in \T$. Thus, $\T$ is normal in~$\H$.

The following lemma yields a proof of the above proposition.

\begin{lemma}
	\begin{enumerate}
		\item There is an isomorphism $\T\cong \Z^{n-1}.$
		\item	The short exact sequence 
				$1\to \T \to \H \to \H/\T \to 1$ splits. 
		\item There is an isomorphism $\H/\T \cong \H^0.$
	\end{enumerate}
\end{lemma}
	
\begin{proof}
	\begin{enumerate}
		\item
			The group $\T$ has a representation $\langle \tau_1, \dots, \tau_{n-1}~|~ \tau_i \tau_j = \tau_j \tau_i ~\forall i,j\rangle$.
		\item
			The map $\H \to \Z^{n-1}$ sending $M$ to $(\bar\mu_2(M)_1,\ldots,\bar\mu_2(M)_{n-1})$ gives rise to left exactness.  
		\item
			There is a homomorphism $\H/\T \to \H^0$ which maps $[M]$ to $MI_\tau$ where $\tau=\tau_{n-1}^{t_{n-1}}\circ \cdots \circ \tau_1^{t_1}$ with $t_i=\bar\mu_2(M)_i$. It is an isomorphism.
	\end{enumerate}
\end{proof}

\subsection{Combination of the two filtrations}
\label{subsection:refine}
Let $\H^0_{g,n}[k] := \H^0_{g,n} \cap \H_{g,n}[k]$. For each $k$, $\H[k]\cong \Z^{n-1} \times \H^0[k]$. In fact, $\H^0[k]=H(2)\cap\H[k]=\H(k-1)\cap\H[k]$, so $\H(k)\supseteq\H^0[k+1]$. 
Also we can define $\H^0_{g,n}(k) := \H^0_{g,n} \cap \H_{g,n}(k)$, but then $\H(k)=\H^0(k)$ for every $k\geq 2$. Thus, for simplicity, we reset $\H(1):= \H^0$. Then $\H(k)=\H^0(k)$ for all $k$, and we obtain a combined descending filtration: 
$$\H\rhd \H^0=\H^0[1] =\H(1)\rhd \H^0[2] \rhd \H(2)\rhd \cdots\rhd \H(k-1) \rhd \H^0[k] \rhd \H(k)\rhd \H^0[k+1] \rhd \cdots  $$ 
Note that $\H(k-1)=\H^0[k]$ only if $g=0$, and $\H[k]=\H(k)$ only if $n=1$.
All $\H^0$, $\H[k]$, $\H^0[k]$, and $\H(k)$ are normal subgroups of $\H$.

The invariant $\tilde\mu_k$ can be divided as follows for $k\geq 2$:
$$\mu_{k}\colon \frac{\H^0[k]}{\H(k)} \hooklongrightarrow  (F_{k-1}/F_{k})^{n-1}$$
$$(\mu_{k}',\mu_{k}'') \colon \frac{\H(k-1)}{\H^0[k]} \hooklongrightarrow  (F_{k-1}/F_{k})^{2g}$$
In the case of $k=2$, the bottom map is not a homomorphism, but a crossed homomorphism with the action of $\tilde\eta_2$ on $H$. The others are all homomorphisms. 
The two subquotients of $\H$ are also finitely generated free abelian except $\frac{\H(1)}{H^0[2]}$, and hence there are isomorphisms
$$ \frac{\H^0[k]}{\H^0[k+1]} \cong \frac{\H^0[k]}{\H(k)}\times \frac{\H(k)}{\H^0[k+1]},\quad  \frac{\H(k)}{\H(k+1)} \cong \frac{\H(k)}{\H^0[k+1]}\times \frac{\H^0[k+1]}{\H(k+1)}$$
for $k\geq 2$, which are not canonical.
While $\H(k)=\Ker(\mu_k,\mu_k',\mu_k'')$ in $\H^0$, $\H^0[k]=\Ker(\mu_{k-1},\mu_k',\mu_k'')$ in $\H^0$.
From now on, we will investigate the four quotient groups $\frac{\H(k)}{\H(k+1)}, \frac{\H^0[k]}{\H^0[k+1]}, \frac{\H(k)}{\H^0[k+1]}$, and~$\frac{\H^0[k+1]}{\H(k+1)}$ for~$k\geq 1$.
\begin{remark}
	\begin{enumerate}
	\item In the author's paper \cite[page~920]{So}, $\H^0$ defined as $\{M\in\H~|~\mu_2(M)=1\}=:\H^0_{old}$ is different from $\H^0$ here; The old one is smaller than our new $\H^0$ unless $n=1$. Though $\H^0_{old}\cap\H[k]=\H^0[k]$ for $k\geq 3$, $\H^0_{old}\cap \H[2]=\H(2)\subsetneq \H^0[2]$ for $n>1$. The new definition is more suitable for the concept of ``zero-framed'' homology cylinder since it corresponds to 0-framed string link in homology 3-balls in the case of $g=0$.
	\item There are some errors related to $\tilde\mu_2$ in \cite{So}. On \cite[Corollary~3.8 (4)]{So}, $\tilde\mu_2$ is not a homomorphism on $\H$ or $\Ker\mu_2$ though it is a homomorphism on $\H[2]$. \cite[Theorem~4.3]{So} holds only for $q\geq 3$, that is, the map $\tilde\mu_2\colon \H/\H(2) \hookrightarrow (F/F_2)^{2g+n-1}$ is not a homomorphism. It is only a crossed homomorphism. 
	\end{enumerate}
\end{remark}

\section{Images and quotients of the filtration under the invariants} \label{sec:image}
\subsection{Images under extended Johnson homomorphism and $\Aut(F/F_k)$}
\label{subsec:result_Johnson}
Clarification of the images of the filtration under the extended Johnson homomorphisms provide a relationship between the filtration and a filtration of an automorphism group of a free nilpotent group.
Before we proceed, we introduce new notations as follows:
\begin{align*}
		\Aut_*(F/F_l)&\cong\bigg\{\phi\in \Aut(F/F_l)~\bigg|\begin{array}{ll}\textrm{there is a lift }\tilde\phi \in \Aut(F/F_{l+1}) \textrm{ such that } \\
			  \tilde\phi(x_i)=\alpha_i^{-1} x_i \alpha_i \textrm{ for some }\alpha_i \in F/F_l\textrm{ and }\tilde\phi \textrm{ fixes } [\partial_n]\end{array}\bigg\} \\
	\K_{l,k} &=\Ker \{\Aut_*(F/F_l)\to \Aut_*(F/F_k)\} \qquad (l\geq k) \\
	\A_{l,k} &=\{\phi \in \Aut_*(F/F_l)~|~\phi_k(m_j)=m_j, \phi_k(l_j)=l_j, \phi_{k+1}(x_i)=x_i \} \\
				&= \{ \phi \in  \K_{l,k}~|~ \phi_{k+1}(x_i)=x_i\}	\qquad (l>k)
\end{align*}
Here, for an automorphism $\phi$ of $F/F_l$ and $t \leq l$, $\phi_t$ denotes the automorphism of $F/F_t$ induced from $\phi$. Then, they form a descending filtration of $\Aut_*(F/F_l)$
$$\Aut_*(F/F_l)=\K_{l,1}=\A_{l,1}\rhd \K_{l,2} \rhd \A_{l,2}\rhd \cdots\rhd \A_{l,k-1}\rhd \K_{l,k} \rhd \A_{l,k}\rhd \K_{l,k+1} \rhd \cdots  .$$ 
The subgroups $\A_{l,k}$ and $\K_{l,k}$ are normal in $\Aut_*(F/F_l)$.
Remark that the lift $\tilde\phi$ can be chosen to be an endomorphism in the definition of $\Aut_*(F/F_l)$ by the following proposition.

\begin{proposition}
	Suppose $\phi\in\Aut(F/F_k)$ with $k\geq 2$. Then any lift of $\phi$ to an endomorphism of $F/F_l$ is an automorphism for $l>k$.
\end{proposition}

\begin{proof}
	The argument comes from the proof of \cite[Theorem~2.1]{And}. Let $\tilde\phi$ be a lift of $\phi$ to an endomorphism of~$F/F_{l}$. Then $(\Im\tilde\phi) \cdot (F_k/F_{l}) = F/F_{l}$. Since $F/F_{l}$ is nilpotent, $\Im\tilde\phi= F/F_{l}$ (see \cite[Corollary~10.3.3]{Hall}). Therefore, $\tilde\phi$ is onto. Since $F/F_{l}$ is Hopfian, $\tilde\phi$ is an automorphism.
\end{proof}
Therefore, $\Aut_*(F/F_l)$ is a generalization of both $\Aut_0(F/F_l)$ of the case $n=1$ and $\Aut_1(F/F_l)$ of the case $g=0$. Now we relate the combined filtration of $\H$ in Section~\ref{subsection:refine} to a filtration of the automorphism group of~$F/F_l$ by determining the images of the subgroups in the filtration of $\H$ under the extended Johnson homomorphisms.

\begin{theorem}
	\label{theorem:surjections}
	Suppose $l>k\geq 1$.
	\begin{enumerate}
		\item The composition $\tilde\eta^0_l\colon \H^0_{g,n}\hookrightarrow \H_{g,n} \xrightarrow{\tilde\eta_l} \Aut_*(F/F_l)$ is surjective.
		\item The inverse image $(\tilde\eta_l^0)^{-1}(\K_{l,k})=\H^0_{g,n}[k]$. Thus, $\H^0[k]\to \K_{l,k}$ is surjective.
		\item The inverse image $(\tilde\eta_l^0)^{-1}(\A_{l,k})=\H_{g,n}(k)$. Thus, $\H(k)\to \A_{l,k}$ is surjective.
	\end{enumerate}	  
\end{theorem}

We postpone the proof to the next section. From this theorem, we obtain the following directly.

\begin{corollary}
\label{corollary:surjectivity}
 	For $l>k \geq 1$, there are isomorphisms
	$$ \frac{\H(k)}{\H(k+1)} \cong \frac{\A_{l+1, k}}{\A_{l+1, k+1}}, \quad
	 \frac{\H^0[k]}{\H^0[k+1]} \cong \frac{\K_{l, k}}{\K_{l, k+1}}\cong \K_{k+1,k}, $$
	$$	\frac{\H^0[k+1]}{\H(k+1)} \cong \frac{\K_{l+1, k+1}}{\A_{l+1, k+1}}, \quad  
		\frac{\H(k)}{\H^0[k+1]}\cong \frac{\A_{l, k}}{\K_{l, k+1}}\cong \A_{k+1,k}.$$
\end{corollary}

\subsection{Images under extended Milnor invariants and free Lie algebras}
We consider the restricted invariant $\tilde\mu_{k+1}\colon \H(k) \to (F_k/F_{k+1})^{2g+n-1}$, which is a homomorphism for each $k\geq 2$.
Recall that $L_k(H)$ is the degree $k$ part $\L_k(H)$ of the free Lie algebra over $H$, and $\D_k(H)$ is the kernel of the bracket map $H\otimes \L_k(H)\to \L_{k+1}(H)$.
Note that $\L_k(H)$ is isomorphic to~$F_k/F_{k+1}$.
We consider the image of the homomorphism invariant $\tilde\mu_{k+1}\colon \H(k) \to (F_k/F_{k+1})^{2g+n-1}$.
We identify the codomain $(F_k/F_{k+1})^{2g+n-1}$ with $H\otimes \L_k(H)$ along
$$ (\alpha_1,\ldots,\alpha_{n-1},\beta_1,\ldots,\beta_g, \gamma_1,\ldots,\gamma_g) \mapsto \sum_i x_i\otimes \alpha_i+ \sum_j(m_j\otimes \gamma_j - l_j \otimes \beta_j).$$ 
We define a surjective homomorphism $\p_k$ as follows.
\begin{align*}
	 \p_k\colon (F_k/F_{k+1})^{2g+n-1}&\to F_{k+1}/F_{k+2} \\
	(\alpha_1,\ldots,\alpha_{n-1},\beta_1,\ldots,\beta_g, \gamma_1,\ldots,\gamma_g)&\longmapsto \prod_{i=1}^{n-1} [x_i,\alpha_i] \prod_{j=1}^{g} [m_j,\gamma_j][\beta_j,l_j].
\end{align*}
Then, $\Ker\p_k$ is identified with $\D_k(H)$ along the above identification.

\label{splitting}
We can think $\Sigma_{g,n}$ as a boundary connected sum of $\Sigma_{0,n}$ and $\Sigma_{g,1}$ with $\pi_1(\Sigma_{0,n})=\langle x_1,\ldots, x_{n-1} \rangle =: F'$ and $\pi_1(\Sigma_{g,1})=\langle m_1,\ldots,m_g, l_1,\ldots,l_g\rangle=: F''$. Let $H':=H_1(\Sigma_{0,n})$ and $H'':=H_1(\Sigma_{g,1})$. Then $F=F'\ast F''$ and $H=H'\oplus H''$.
 
\begin{lemma}
\label{lemma:D'}
	For every $k\geq 1$,
	\begin{align*}
		\Ker\p_k \cap \{(F_k/F_{k+1})^{n-1}\times 0^{2g}\} &= \Ker\p_k \cap \{(F'_k/F'_{k+1})^{n-1}\times 0^{2g}\}\\ 
		&= \Ker\{\p'_k\colon(F'_k/F'_{k+1})^{n-1}\rightarrow F'_{k+1}/F'_{k+2}\} \\
	\end{align*}
	where $\p'_k$ sends $(\alpha_1,\ldots,\alpha_{n-1})$ to  
$\prod_{i=1}^{n-1} [x_i,\alpha_i]$. In other words,
$$\D_k(H')\cong \Ker\{H' \otimes \L_k(H) \rightarrow \L_{k+1}(H)\} \subset \D_k(H).$$
\end{lemma}

We clarify the image of the extended Milnor invariants as follows.

\begin{theorem} \label{theorem:Milnor} 
For $k\geq2$, the image of $\tilde\mu_{k+1}\colon \H(k) \to (F_k/F_{k+1})^{2g+n-1}$
 is $\Ker\p_k\cong \D_k(H)$.
 In addition, the inverse image of $\Ker\p'_k \cong \D_k(H')$ is $\H^0[k+1]$.
\end{theorem}

Now we relate the successive quotients of the filtration of $\H_{g,n}$ to~$\D_k(H)$.

\begin{theorem}
	\label{theorem:rank}
	For $k\geq 2$, there are isomorphisms 
	$$ \frac{\H(k)}{\H(k+1)} \cong \D_k(H), \quad 
	 \frac{\H^0[k+1]}{\H(k+1)} \cong \D_k(H'), \quad \textrm{and}\quad
	 \frac{\H[2]}{\H(2)} \cong \D_1(H').$$
\end{theorem}
The proof of the above theorems are achieved by associating the image $\A_{k+2, k}$ of $\H(k)$ under $\tilde\eta_l$ with $\Ker\p_k\cong\D_k(H)$. We also postpone the precise proof to the next section.

\begin{figure}
$$
\tikzcdset{arrow style=tikz, diagrams={>=stealth}}
\tikzcdset{every label/.append style = {font = \tiny}}
\begin{tikzcd}[row sep={1.5cm,between origins}, column sep={2.4cm,between origins}]
&&[-4mm] \H \ar[from=dl, hook,"\begin{array}{l}
\Coker\\= \Z^{n-1}\end{array}" xshift=0, sloped] \ar[from=ddd, hook]\\[-5mm]
& \H^0 \ar[from=ddd, hook] \ar[rrr, twoheadrightarrow, crossing over] &[-10mm] &[-2mm]& \K_{l,1}  \ar[from=ull,twoheadrightarrow, "\tilde\eta_l"]\ar[from=ddd, hook, "
\Coker=\K_{2,1}" {fill=white,xshift=0,yshift=-12}, sloped] \rlap{$=\Aut_*(F/F_l)$}\\
\H(1) \ar[from=ddd, hook] \ar[ru, equal] \ar[rrr, twoheadrightarrow, crossing over, "\tilde\eta_l^0"] &&& \A_{l,1} \ar[ru, equal]\\[-.8cm]
&& \H[2] \ar[from=ddd, hook] \ar[from=dl, hook,"\begin{array}{l}
\Coker\\= \Z^{n-1}\end{array}" xshift=0, sloped]  \\[-5mm]
& \H^0[2] \ar[from=ddd, hook] \ar[rrr,twoheadrightarrow, crossing over] \ar[uul, hook'] &&& \K_{l,2} \ar[from=ull,twoheadrightarrow]\ar[from=ddd, hook, "
\Coker=\K_{3,2}" {fill=white,xshift=0,yshift=-12}, sloped] \ar[luu, hook', "
\Coker= \A_{2,1}"' {
yshift=11}
, sloped]\\
\H(2) \ar[from=ddd, hook] \ar[ru, hook, crossing over] \ar[rrr, twoheadrightarrow, crossing over] &&& \A_{l,2} \ar[uuu, hook, crossing over] \ar[ru, hook, "\begin{array}{l}
\Coker\\=\frac{\D_1(H')}{\Z^{n-1}}\end{array}" {xshift=-2}
, sloped] \\[-.8cm]
&& \H[3] \ar[from=ddd, hook, dash pattern=on 10mm off 1mm on 1mm off 1mm on 1mm off 1mm on 1mm off 1mm on 1mm off 1mm on 1mm off 1mm on 1mm off 1mm on 1mm off 1mm on 11mm] \ar[from=dl, hook,"\begin{array}{l}
\Coker\\= \Z^{n-1}\end{array}" xshift=0, sloped]\\[-5mm]
& \H^0[3] \ar[from=ddd, hook, dash pattern=on 10mm off 1mm on 1mm off 1mm on 1mm off 1mm on 1mm off 1mm on 1mm off 1mm on 1mm off 1mm on 1mm off 1mm on 1mm off 1mm on 11mm] \ar[uul, hook'] \ar[rrr, twoheadrightarrow, crossing over] &&& \K_{l,3} \ar[from=ddd, hook, dash pattern=on 10mm off 1mm on 1mm off 1mm on 1mm off 1mm on 1mm off 1mm on 1mm off 1mm on 1mm off 1mm on 1mm off 1mm on 1mm off 1mm on 11mm] \ar[from=ull,twoheadrightarrow] \ar[luu, hook', "
\Coker= \A_{3,2}"' {xshift=0,yshift=11}, sloped]\\
\H(3) \ar[from=ddd, hook, dash pattern=on 10mm off 1mm on 1mm off 1mm on 1mm off 1mm on 1mm off 1mm on 1mm off 1mm on 1mm off 1mm on 1mm off 1mm on 1mm off 1mm on 11mm] \ar[ru, hook, crossing over] 
&&& \A_{l,3} \ar[from=lll, twoheadrightarrow, crossing over] \ar[uuu, hook, crossing over,"
\Coker= \D_2(H)" {fill=white,xshift=0,yshift=0}, sloped] \ar[ru, hook, "\begin{array}{l}
\Coker\\= \D_2(H')\end{array}" {xshift=-2}, sloped] \\[.3cm]
&&\H[k] \ar[from=ddd, hook]  \ar[from=dl, hook, "\begin{array}{l}
\Coker\\= \Z^{n-1}\end{array}" xshift=0, sloped] \\[-5mm]
\hphantom{} & \H^0[k] \ar[from=ddd, hook]\ar[rrr, twoheadrightarrow, crossing over] &&\hphantom{} & \K_{l,k}\ar[from=ull,twoheadrightarrow] \ar[from=ddd, hook, "
\Coker=\K_{k+1,k}" {fill=white,xshift=0,yshift=-12}
, sloped]
\\
\H(k) 
\ar[from=ddd, hook] \ar[ru, hook, crossing over] \ar[rrr, twoheadrightarrow, crossing over] &&& \A_{l,k}  \ar[uuu, hook, dash pattern=on 10mm off 1mm on 1mm off 1mm on 1mm off 1mm on 1mm off 1mm on 1mm off 1mm on 1mm off 1mm on 1mm off 1mm on 1mm off 1mm on 11mm, crossing over] \ar[ru, hook] \ar[ru, hook, "\begin{array}{l}
\Coker\\= \D_{k-1}(H')\end{array}" {xshift=-2}, sloped]\\[-.8cm]
&& \H[k+1] \ar[from=ddd, dash pattern=on 1mm off 1mm on 1mm off 1mm on 1mm off 1mm on 1mm off 1mm on 1mm off 1mm on 1mm off 1mm on 15mm]  \ar[from=dl, hook,"\begin{array}{l}
\Coker\\=\Z^{n-1}\end{array}" xshift=-2, sloped]\\[-5mm]
& \H^0[k+1] \ar[from=ddd, dash pattern=on 1mm off 1mm on 1mm off 1mm on 1mm off 1mm on 1mm off 1mm on 1mm off 1mm on 1mm off 1mm on 15mm]\ar[uul, hook']   \ar[rrr, twoheadrightarrow, crossing over] &&& \K_{l,k+1} \ar[from=ull,twoheadrightarrow]\ar[from=ddd, dash pattern=on 1mm off 1mm on 1mm off 1mm on 1mm off 1mm on 1mm off 1mm on 1mm off 1mm on 1mm off 1mm on 15mm] \ar[luu, hook', "
\Coker= \A_{k+1,k}"' {xshift=0,yshift=11},
sloped]\\
\H(k+1) \ar[from=ddd, dash pattern=on 1mm off 1mm on 1mm off 1mm on 1mm off 1mm on 1mm off 1mm on 1mm off 1mm on 1mm off 1mm on 15mm] \ar[ru, hook, crossing over] \ar[rrr, twoheadrightarrow, crossing over] &&& \A_{l,k+1}  \ar[from=ddd, dash pattern=on 1mm off 1mm on 1mm off 1mm on 1mm off 1mm on 1mm off 1mm on 1mm off 1mm on 1mm off 1mm on 15mm] \ar[uuu, hook, crossing over, "
\Coker= \D_k(H)" {fill=white,xshift=0,yshift=0}, sloped] \ar[ru, hook,"\begin{array}{l}
\Coker\\= \D_{k}(H')\end{array}" {xshift=-2}, sloped]\\[-.8cm]
&& \phantom{A} \\[-5mm]
& \phantom{A} \ar[uul, dash pattern= on 1mm off 1mm on 1mm off 1mm on 1mm off 1mm on 1mm off 1mm on 1mm off 1mm on 15mm] & && \phantom{A}\ar[uul, dash pattern= on 1mm off 1mm on 1mm off 1mm on 1mm off 1mm on 1mm off 1mm on 1mm off 1mm on 15mm]\\
 \phantom{A}  &&& \phantom{A} \\[-.8cm]
\end{tikzcd}$$
\caption{}
\label{figure:summary}
\end{figure}

The diagram in Figure~\ref{figure:summary} on page \pageref{figure:summary}, summarizes results so far. The horizontal surjections are (restrictions of) the extended Johnson homomorphisms. What written over the injection arrows are the cokernels of the maps. 

In the above diagram, each subgroup of $\H$ on the left surjects to the subgroup of $\Aut_*(F/F_l)$ located at the same position on the right under $\tilde\eta_l$. What written on the injection arrows are the cokernels of the maps. The cokernels of the maps at the same position on both sides are isomorphic except $\Z^{n-1}$ which corresponds to the ``framing'' $\bar\mu_2(-)_i$.

\subsection{Images under extended Orr invariants}
We can verify the image of the extended Orr invariants and relation between the quotients of the filtration and $H_3(F/F_k)$ as follows. Note that the injection $F'\to F$ induces $H_3(F'/F'_k)\to H_3(F/F_k)$, which is also injective. By abuse of notation, we think $H_3(F'/F'_k)$ as a subgroup of $H_3(F/F_k)$.
\begin{theorem}
	For $k\geq2$, $\tilde\theta_k\colon\H(k) \to H_3(F/F_k)$ is surjective.
	The inverse image of $\Im\{H_3(F/F_l)\to H_3(F/F_k)\}$  under $\tilde\theta_k$ is $H(l)$ for $k\leq l\leq 2k-1$.
\end{theorem}

To prove the surjectivity, we use the fact that $H_3(G)$ is isomorphic to the $k$-dimensional oriented bordism group $\Omega_k(G)$ of~$K(G,1)$ for any group~$G$. Any element of $(X^3 \to K(F/F_k,1)) \in \Omega_3(F/F_k) \cong H_3(F/F_k)$ can be realized as a homology cylinder with trivial~$\tilde\mu_k$, which is obtained by cutting $X$ along an embedded $\Sigma_{g,n}$. The surface should be chosen so that $\pi_1(\Sigma_{g,n})\to \pi_1(X)\to F/F_k$ sends the generetors $x_i,m_j,l_j$ to $[x_i],[m_j],[l_j]$ in~$F/F_k$. 

\begin{corollary} \label{cor:Orr}
	For $k\geq2$, the invariant $\tilde\theta_k$ induces an isomorphism
	$$\frac{\H(k)}{\H(k+1)} \to \Coker\{H_3(F/F_{k+1})\xrightarrow{\psi_{k+1,k}} H_3(F/F_k)\}.$$
\end{corollary}

We can obtain them in two ways. One approach is to apply the known fact that $H_3(F/F_k)\cong\displaystyle\bigoplus_{i=k}^{2k-2} \Z^{(2g+n-1)N_k - N_{k+1}}$ and $\D_k(H) \cong \Coker\{H_3(F/F_{k+1})\to H_3(F/F_k)\}$ for $k\geq 2$ (see~\cite{IO}). Alternatively, we can prove them directly. 
We leave the proof to the reader.  

The corollary implies that the invariants $\tilde\theta_k$ with values in the cokernel and $\tilde\mu_{k+1}$ on $\H(k)$ are equivalent. 
The invariant $\tilde\theta_k$ with values in $H_3(F/F_k)$ is equivalent to $\tilde\mu_{2k-1}$ on~$\H(k)$. 
We have an isomorphism $\tilde\theta_k\colon \frac{\H(k)}{\H(2k-1)}\to H_3(F/F_k)$. Remark that $\tilde\mu_{2k-1}\colon \H(k)\to (F_k/F_{2k-1})^{2g+n-1}$ is also a homomorphism. It follows from Lemma~\ref{lemma:keys}(3).

\subsection{Relation between the invariants} \label{subsec:relation}
\[\tikzcdset{arrow style=tikz, diagrams={>=stealth}}
\begin{tikzcd}[row sep={1.5cm,between origins}, column sep={1.8cm,between origins}]
 &[-2.8mm]&[2mm]&1 \ar[dd]&&[2mm]&[-2.8mm]\\[-11mm]
1\ar[dr]&&&&&& 1 \ar[dl]\\[-9mm]
&\H(2k-1) \ar[hook]{rr} \ar[hookrightarrow]{drr} &&\H(k+1) \ar[hook]{rr} \ar[hookrightarrow]{d} &&\H^0[k+1] \ar[hook']{lld} \\
&&& \H(k) \ar[lld, two heads, "\tilde\eta_{k+1}"']  \ar[d, two heads, "\tilde\mu_{k+1}"] \ar[rrd, two heads,"\tilde\theta_k"] \\
&\A_{k+1,k} \ar[ddd, hook',start anchor={[xshift=-5]},end anchor={[xshift=-10]}] && \Ker \p_k \ar[ddd,hook', start anchor={[xshift=-5]},end anchor={[xshift=-10]}] \ar[ll, two heads]&\cong \D_k(H)\hphantom{H\otimes\;} \ar[ddd, hook',start anchor={[xshift=-5]},end anchor={[xshift=-10]}] & H_3(F/F_k) \ar[l, two heads, end anchor={[xshift=-5ex]}] \ar[dd, two heads] \\[-9mm]
1\ar[from=ur]&&&&&& 1 \ar[from=ul]\\[-11mm]
 &&&1\ar[from=uu] &&\hphantom{aaaaaaa}\Coker\psi_{k+1,k}
\ar[luu, "\sim" {xshift=7, sloped},end anchor={[xshift=-12]}] \\[-10mm]
 &\Aut(F/F_{k+1})\hphantom{aaaaaa} && (F_k/F_{k+1})^{2g+n-1}\hphantom{aaaaaaaaa} \ar[ll,end anchor={[xshift=-30]}] &\qquad\qquad\cong H\otimes \L_k(H)\hphantom{aaaaaaaaaa}
 \end{tikzcd}\]

The map $\psi_{k+1,k}$ is the natural map $H_3(F/F_{k+1})\to H_3(F/F_k)$ appeared in Corollay~\ref{cor:Orr}.
The above diagram commutes, and the diagonal and vertical sequences are exact.
The bottom left map $(F_k/F_{k+1})^{2g+n-1}\to \Aut(F/F_{k+1})$ sends $(\alpha_1,\ldots,\alpha_{n-1},\beta_1,\ldots,\beta_g, \gamma_1,\ldots,\gamma_g)
$ to the automorphism $x_i \mapsto \alpha_i^{-1} x_i \alpha_i$, $m_j\mapsto m_j\beta_j$, and $l_j\mapsto l_j\gamma_j$, and its image is in $\Ker\{\Aut(F/F_{k+1})\to\Aut(F/F_k)\}$. The map and its restriction on $\Ker\p_k$ are not homomorphisms. All the other arrows in the above diagram are homomorphisms.
The lower right homomorphism $H_3(F/F_k)\to \D_k(H)$ is defined using the Massey product of classes in $H^1(F/F_k)$ in~\cite{GL}.

\subsection{Rank of quotients of the filtration} 
From Theorem~\ref{theorem:rank}, we get the exact ranks of quotient groups:

\begin{corollary}
	\label{corollary:rank} Suppose $k\geq 2$, then
	$$\frac{\H(k)}{\H(k+1)}\cong \Z^{(2g+n-1)N_k - N_{k+1}}$$
	where $N_k=\rank F_k/F_{k+1}=\frac{1}{k}\sum_{d|k} \varphi(d)\; (2g+n-1)^{k/d}$, and $\varphi$ is the M\"obius function. Also, we have
	$$\frac{\H^0[k+1]}{\H(k+1)}\cong \Z^{(n-1)N_k' - N_{k+1}'} \quad \textrm{and} \quad \frac{\H^0[2]}{\H(2)}\cong \Z^{\binom{n-1}{2}}= \Z^{\frac{1}{2}(n-1)(n-2)}$$
		where $N_k'= \rank F'_k/F'_{k+1}=\frac{1}{k}\sum_{d|k} \varphi(d)\; (n-1)^{k/d}$.
\end{corollary}
We can also compute the ranks of $\frac{\H(k)}{\H^0[k+1]}$ and $\frac{\H^0[k]}{\H^0[k+1]}\cong \frac{\H[k]}{\H[k+1]}$ for~$k\geq 2$ using the above corollary.

\subsection{First quotient of the extended Milnor filtration} \label{subsec:first level}
In order to complete the investigation, it remains to characterize $\frac{\H(1)}{\H(2)}$ and $\frac{\H^0[1]}{\H^0[2]} \cong \frac{\H(1)}{\H^0[2]}$, which are all the non-abelian cases except~$\frac{\H_{0,n}(1)}{\H_{0,n}(2)}$ .
We found that $\frac{\H(1)}{\H^0[2]}$ is isomorphic to $\A_{2,1}=\Aut_*(F/F_2)$. Recall that $\H/\H[2] \cong \Aut^*(H)$ where 		
	\begin{align*}
		\Aut^*(H)&=\bigg\{\phi\in\Aut(H) ~\bigg|\begin{array}{ll}\phi \textrm{ fixes }[\partial_i]\textrm{ for all }i=1,\cdots, n \textrm{ and }\\
		 \textrm{preserves the intersection form of } \Sigma\end{array}\bigg\}
	\end{align*}
	(see \cite[Proposition~2.3 and Remark~2.4]{GS}). Note that 
	\begin{align*}
		\Aut^*(H)&\cong\bigg\{\begin{bmatrix}
			I_{n-1} & A \\ 0 & P 
		\end{bmatrix} ~\bigg|~P\in\Sp(2g,\Z), A\in{\bf{M}}_{(n-1)\times 2g}(\Z) \bigg\} \\
		&\cong\Z^{(n-1)2g} \rtimes \Sp(2g,\Z)
	\end{align*}
	where ${\bf{M}}_{n-1,2g}(\Z)$ is the additive group of $(n-1)\times 2g$ matrices over $\Z$, and $\Sp(2g,\Z)$ is the symplectic group over~$\Z$.
So, $\frac{\H(1)}{\H^0[2]} \cong \frac{\H^0[1]}{\H^0[2]}\cong\Z^{(n-1)2g} \rtimes \Sp(2g,\Z)$. The fact $\Aut^*(H)=\Aut_*(F/F_2)$ can be also checked directly. 
In the remaining part of this section, we focus on $\frac{\H(1)}{\H(2)}$. Let us discuss whether the exact sequence
$$1\to \frac{\H^0[2]}{\H(2)}\to \frac{\H(1)}{\H(2)} \to \frac{\H(1)}{\H^0[2]}\to1$$
splits. 
	We remark that if $g=0$, i.e.~the case of string links in homology 3-balls, 
	$\H_{0,n}(1)=\H^0_{0,n}[2]$, so $$\frac{\H_{0,n}(1)}{\H_{0,n}(2)}=\frac{\H^0_{0,n}[2]}{\H_{0,n}(2)}\cong  \Z^{\binom{n-1}{2}}=\Z^{\frac{1}{2}(n-1)(n-2)}.$$
This corresponds to the mutual linking while $\H/\H(1) \cong \Z^{n-1}$ corresponds to the self-linking.
	If $n=1$, then $\H_{g,1}(1)=\H_{g,1}$ and $\H_{g,1}[k]=\H_{g,1}^0[k] =\H_{g,1}(k)$, so $$\frac{\H_{g,1}(1)}{\H_{g,1}(2)}=  \frac{\H_{g,1}(1)}{\H^0_{g,1}[2]} = \frac{\H_{g,1}}{\H_{g,1}[2]} \cong \Aut_*(H)\cong \Sp(2g, \Z).$$
	 For general $g$ and $n$, we prove a partial result.
We recall $H''=H_1(\langle m_1,\cdots, m_g, l_1,\cdots, l_g \rangle) \subset H$ (see page~\pageref{splitting}). Let
$\H^{0}[1.5] := \{M\in\H^{0}~|~\tilde\eta_2(M)(H'')\subset H''\}$. Then, $M\in \H[1.5]$ if and only if 
$\tilde\eta_2(M)$ corresponds to the matrix $\left[\begin{array}{cc} I & 0 \\0 &P \end{array}\right]$. 
So, $\frac{\H^{0}[1.5]}{\H^{0}[2]} \cong \Sp(2g,\Z)$. However, $\H^{0}[1.5]$ is not normal in~$\H^{0}$.
\begin{proposition}
	There is a split exact sequence
	$$ 1\to \frac{\H^0[2]}{\H(2)} \to \frac{\H^0[1.5]}{\H(2)} \to \frac{\H^0[1.5]}{\H^0[2]}\to 1.$$ Therefore, 
	\begin{align*} 
		\frac{\H_{g,n}^0[1.5]}{\H_{g,n}(2)} &\cong \Z^{\binom{n-1}{2}} \times \Sp(2g,\Z)~ \cong \frac{\H_{0,n}(1)}{\H_{0,n}(2)}\times \frac{\H_{g,1}(1)}{\H_{g,1}(2)} ,
		\\
		\frac{\H_{g,n}[1.5]}{\H_{g,n}(2)} &\cong \D_1(H')\times \Sp(2g,\Z) \cong \frac{\H_{0,n}}{\H_{0,n}(2)}\times \frac{\H_{g,1}}{\H_{g,1}(2)}.
	\end{align*}
\end{proposition}

\begin{proof}
	We define a map $f\colon \frac{\H^0[1.5]}{\H(2)} \to \frac{\H^0[2]}{\H(2)}$ so that 
	\begin{enumerate}
		\item $\mu_2'(f(M)), \mu_2''(f(M))$ are trivial, 
		\item $\prod_i [x_i,\mu_2(f(M))_i]=1$ in $F_2/F_3$, and 
		\item for each $i$, $\mu_2(f(M))_i^{-1} \mu_2(M)_i$ is represented by a product of $x_r,m_j,l_j$ with $r>i, 1\leq j \leq g$.
	\end{enumerate}
		
	Since $\prod_i [x_i,\mu_2(f(M))_i^{-1} \mu_2(M)_i]=\prod_i [x_i,\mu_2(f(M))_i^{-1}] [x_i,\mu_2(M)_i]=\prod_i [x_i,\mu_2(M)_i]$ is given and $\{[x_i, x_r],[x_i, m_j],[x_i,l_j]~|~ r>i, 1\leq j\leq g\}$ is a subset of a Hall basis for $F_2/F_3$, which is a basis for a subgroup of $F_2/F_3$ generated by $\{[x_i, z]~|~z\in F/F_2\}$, $\mu_2(f(M))_i^{-1} \mu_2(M)_i$ is uniquely determined as a product of $x_r, m_j, l_j$ for $r>i, 1\leq j \leq g$. Hence so is~$\mu_2(f(M))_i$. An element of $\frac{\H^0[2]}{\H(2)}$ is determined by $\mu_2(-)$. Therefore, $f$ is well-defined.
	
	To see that the map is a homomorphism, we should check that $\mu_2(f(M)f(N))=\mu_2(f(MN))$. 	It is enough to show that $\mu_2(f(M)f(N))_i^{-1} \mu_2(M)_i$ is represented by a product of $x_r,m_j,l_j$ for $r>i, 1\leq j \leq g$. In $\H^0[2]$, $\mu_2$ is a homomorphism.
	\begin{align*}
		\mu_2(f(M)f(N))_i^{-1} \mu_2(MN)_i &=\mu_2(f(N))_i^{-1}\mu_2(f(M))_i^{-1}\mu_2(M)_i \; \tilde\eta_2(M)(\mu_2(N)_i) \\
		&=\mu_2(f(M))_i^{-1}\mu_2(M)_i \; \mu_2(f(N))_i^{-1}\tilde\eta_2(M)(\mu_2(N)_i)
	\end{align*}
	is a product of $x_r,m_j,l_j$ for $r>i, 1\leq j \leq g$ since $\tilde\eta_2(M)$ preserves the powers of $x_i$ for $M\in\H^0[1.5]$. This homomorphism $f$ is a left splitting, so the given exact sequence splits.
\end{proof}

\begin{figure}
\[
\tikzcdset{arrow style=tikz, diagrams={>=stealth}}
\tikzcdset{every label/.append style = {font = \tiny}}
\begin{tikzcd}[column sep={12mm,between origins}, row sep={6mm,between origins}]
& [-4mm,between origins] 1 \ar{dddr} & [-4mm,between origins] & \Z^{\binom{n-1}{2}}\ar[dddl, equal] &&[4mm,between origins]&[-4mm,between origins] \Z^{\binom{n-1}{2}}\times\Sp(2g,\Z)\ar[ddd,equal] &[-4mm,between origins]&[4mm,between origins]& \Sp(2g,\Z)\ar[dddr,equal] &&[-4mm,between origins] \ar[from=dddl] 1 & [-4mm,between origins] \\ [-2mm,between origins]
1 \ar{ddrr} &&&&&&&&&&&& \ar[from=ddll] 1 \\ [-2mm,between origins] \\[-2mm,between origins]
1 \ar{rr} & &\displaystyle\frac{\H^0[2]}{\H(2)} 
\ar[hookrightarrow]{rrrr} 
\ar[hook]{rrrrdddd}
\ar[hookrightarrow
]{rrdddddd} 
&&&& \displaystyle\frac{\H^0[1.5]}{\H(2)} \ar[dashrightarrow, bend left=15, "\exists \textrm{ splitting}" yshift=-10,sloped]{llll} 
\ar[two heads]{rrrr}
\ar[hook]{dddd}[sloped,xshift=0, yshift=0]{\ntriangleleft}
&&&&\displaystyle\frac{\H^0[1.5]}{\H^0[2]}  
\ar[dashrightarrow, bend left=15, "\exists \textrm{ splitting}" yshift=-10, sloped]{llll}
\ar[dashrightarrow, bend left=15, "\exists \textrm{ splitting}" yshift=-10, sloped]{ddddllll}
\ar[dashrightarrow, bend left=15, "\exists \textrm{ splitting}" yshift=-10, sloped]{ddddddll}
\ar{rr} && 1
\\ \\
\\ \\
&&&&&&
\displaystyle\frac{\H(1)}{\H(2)} \ar[two heads]{uuuurrrr}
\ar[two heads]{ddrr} \ar[dotted, bend left=15, "\nexists \textrm{ splitting}" yshift=-10,sloped]{ddll}
\ar[dotted, bend left=15, "\nexists \textrm{ splitting}" yshift=-10,sloped]{lllluuuu}
\\ \\
&&&&
K \ar[hookrightarrow]{uurr}
\ar[two heads]{rrdddddd} \ar[dotted, bend left=15, "\nexists \textrm{ splitting}" yshift=-10,sloped]{lluuuuuu}
&&&&
\displaystyle\frac{\H(1)}{\H^0[2]} \ar[two heads]{uuuuuurr} \ar[dotted, bend left=15, "\nexists \textrm{ splitting}" yshift=-10,sloped]{uull} \ar[dotted, bend left=15, "\nexists \textrm{ splitting}" yshift=-10,sloped]{ddddddll}
\ar[dr] && \Aut^*(H)\qquad\quad\ar[ll, equal]
\\  &&& 1 \ar[ur] &&&&&& 1
\\ \\  \\
\\ \\
&&&&&&
{\bf{M}}_{(n-1)\times 2g}(\Z) 
\ar[hookrightarrow]{uuuuuurr}
\ar[dotted, bend left=15, "\nexists \textrm{ splitting}" yshift=-10,sloped]{lluuuuuu} &&\Z^{(n-1)2g}\quad \ar[ll,equal]
\\   [6mm,between origins]
&&&&& 1\ar[ur] && 1\ar[from=ul]
\end{tikzcd}
\]
\caption{}
\label{figure:first level}
\end{figure}

To sum up, we have the diagram in Figure~\ref{figure:first level}. Here, $K$ denotes the kernel of $\frac{\H(1)}{\H(2)} \to \frac{\H^0[1.5]}{\H^0[2]}$. All sequences in a straight line are part of short exact sequences that the trivial terms are omitted. But, only those which end with the top right term $\frac{\H^0[1.5]}{\H^0[2]}$ are right split, and the top horizontal sequence is the only one which splits.

\section{Proofs} \label{sec:proofs}
\begin{proof}(Proof of Theorem~\ref{theorem:surjections})
	\begin{enumerate}
		\item
		To prove the surjectivity of the map $\tilde\eta_l\colon \H \to \Aut_*(F/F_l)$, we mainly follow the proof of \cite[Theorem~3]{GL}, and simplify a part of it.
		Suppose $\phi \in \Aut_*(F/F_l)$. We can construct maps $f_\pm \colon \Sigma \to K(F/F_l,1)$ so that $(f_+)_\#\colon \pi_1(\Sigma)=F\to F/F_l$ is the natural surjection and $(f_-)_\#= \phi\circ (f_+)_\#$.
		We may assume $f_+|_{\partial \Sigma} = f_-|_{\partial \Sigma}$ since $f_+|_{\partial_i}$ and $f_-|_{\partial_i}$ are freely homotopic for $i=1,\dots, n-1$ and $f_+|_{\partial_n}$ and $f_-|_{\partial_n}$ are homotopic.
		Let $f=f_+\cup f_-\colon \Sigma \cup_\partial \Sigma \to K(F/F_l,1)$. We claim that there are a compact oriented 3-manifold $M$ bounded by $\Sigma\cup_\partial \Sigma$ and a map $\Phi\colon M\to K(F/F_l,1)$ extending~$f$. The obstruction is $[f]\in \Omega_2(F/F_l)\cong H_2(F/F_l)$. 
		Using a lift $\tilde\phi \in \Aut(F/F_{l+1})$ of $\phi$, we can also consider $\tilde f \colon \Sigma\cup_\partial \Sigma \to K(F/F_{l+1},1)$ as above which factors through~$f$. Since $H_2(F/F_{l+1})\to H_2(F/F_l)$ is the 0-map, $[f]=0 \in H_2(F/F_l)$. Thus, we have $M$ and~$\Phi$.
	
		We will do surgery $M$ to be a homology cylinder over $\Sigma$ by killing $\Ker\{\Phi_*:H_1(M)\to H\}$. As \cite[Lemma~4.6]{GL}, for $\alpha\in\Ker\Phi_*$, there exists $\bar\alpha \in \pi_1(M)$ such that $\bar\alpha \in \Ker\{\Phi_\#\colon \pi_1(M)\to F/F_l\}$ and $\bar\alpha$ represents $\alpha$. Thus, we can surger $M$ along a curve representing $\alpha$.
		Note that the surgery on an embedding $\varphi\colon S^1\times D^2 \to \textrm{(interior of }M)$ induces $H_1(M)/\langle \alpha \rangle \cong H_1(M')/\langle \beta \rangle$ where $\alpha=[\varphi(S^1\times 1)]$, $\beta =[\varphi(1\times S^1)]$, and $M'$ is the resulting manifold \cite[Lemma~5.6]{KM}. 
	
	 As a first step, kill the torsion-free part of $\Ker\Phi_*$. Choose $\alpha \in \Ker\Phi_*$ which maps to a primitive element of $H_1(M,\partial M)$. It is possible since $H_1(M)\cong H \oplus \Ker\Phi_*$ and $\Ker\Phi_* \to H_1(M,\partial M)$ is surjective. Let $C$ be a simple closed curve in the interior of $M$ representing~$\alpha$. Surgery on $C$ results in a new manifold $M'$ with $H_1(M')\cong H_1(M)/\langle\alpha\rangle$. After a sequence of such surgeries, $H_1(M,\partial M)$ becomes a torsion group. By calculating ranks of the groups in the long exact sequence of homology groups of the pair $(M,\partial M)$, we conclude that $\rank H_1(M)=\rank H$ and $\rank \Ker\Phi_* =0$.
	 
	Now $\Ker\Phi_*=\Tor H_1(M) \cong \Tor H_1(M,\partial M)= H_1(M,\partial M)$, so the linking pairing $l\colon \Ker\Phi_* \otimes \Ker\Phi_* \to \Q/\Z$ is non-singular. 
	\begin{lemma}
		Suppose $M$ is a 3-manifold whose linking pairing $\Tor H_1(M)\otimes \Tor H_1(M)\to \Q/\Z$ is non-singular. If $\alpha$ is a torsion element in $H_1(M)$, then there is a surgery along a simple closed curve $C$ representing $\alpha$ which reduces $|\Tor H_1(M)|$, that is, $|\Tor H_1(M')| < |\Tor H_1(M)|$ where $M'$ is the resulting manifold.
	\end{lemma}
	
	More precisely, if $l(\alpha,\alpha)\neq 0$, then there is a normal framing on $C$ such that the order of $\beta$ in $H_1(M')$ is less than that of $\alpha$  in $H_1(M)$, and if $l(\alpha,\alpha)= 0$, then there is a normal framing of $C$ such that the order of $\beta$ in  $H_1(M')$ is infinite. In the latter case, the surgery generates a torsion-free summand. By applying the first step above, we can kill it.
	Hence, after a sequence of surgeries, $\Phi_*\colon H_1(M)\to H$ becomes an isomorphism. 
	
	Let $i_\pm\colon \Sigma\to M$ be the restrictions of $\Sigma\cup_\partial \Sigma \to M$ so that $f \circ i_\pm = f_\pm$. We check that $(M,i_+,i_-)$ is a homology cylinder over $\Sigma$; Clearly, $i_\pm\colon \Sigma\to M$ satisfies $i_+|_\partial=i_-|_\partial$. Since $\Phi_*$ and $f_{\pm*}$ are isomorphisms on $H_1(-)$, so are $i_{\pm*}$ on $H_1(-)$. It follows that $i_\pm$ also induce isomorphisms on $H_2(-)$ from $H_1(M,\partial M)=0$ and $H_2(M)=0$.
	
	To prove the surjectivity of $\H^0\to \Aut_*(F/F_l)$, we will find a 0-framed homology cylinder from the above $M$ which also maps to $\phi$ under $\tilde\eta_l$. Let $t_i:=\bar\mu_2(M)_i \in \Z$ and $\tau=\tau_{n-1}^{t_{n-1}}\circ \cdots \circ \tau_1^{t_1}$ where $\tau_i$ is a Dehn twist along~$\partial_i$. We look at $M':=(M, i_+, i_-\circ \tau)$. For each $i=1,\dots, n-1$, $\mu_2(M')_i = x_i^{-t_i} \mu_2(M)_i$, so $\bar\mu_2(M')_i=0$ and hence $M'\in\H^0$. We have $\tilde\eta_l^0(M')=\tilde\eta_l(M)=\phi$.

	\item
	$(\tilde\eta^0_l)^{-1}(\K_{l,k})=\H^0 \cap \H[k]=\H^0[k]$.
	
	\item
	Suppose $M$ is in $\H^0[k]$. If $\tilde\eta^0_l(M)\in\A_{l,k}$, then $\tilde\eta_{k+1}(M)(x_i)=\mu_k(M)_i^{-1} x_i \mu_k(M)_i = x_i$, hence $[x_i, \mu_k(M)_i]=1$ in $F/F_{k+1}$. It implies $\mu_k(M)_i =x_i^t\in F/F_k$ for some $t\in\Z$. Since $M$ is in $\H^0$, $\mu_k(M)_i$ is trivial. The converse is also true. Hence, $(\tilde\eta^0_l)^{-1}(\A_{l,k}) = \H^0[k] \cap \{M\in\H^0~|~\mu_k(M)=1\} = \H(k)$. 
 	\end{enumerate}
\end{proof}

Before proving Theorem~\ref{theorem:rank}, we present the proof of Lemma~\ref{lemma:D'} and some ingredients.

\begin{proof}(Proof of Lemma~\ref{lemma:D'})
	Naturally, $\Ker \p'_k \to \Ker\p_k \cap \{(F_k/F_{k+1})^{n-1} \times 0^{2g}\}$ is an injective homomorphism. We check that it is an isomorphism. Let $\alpha_i \in F_k/F_{k+1}$ for $i=1,\dots, n-1$ satisfying $\prod_i [x_i,\alpha_i]=1 \in F_{k+1}/F_{k+2}$. We will show that $\alpha_i$ is in $F'_k/F'_{k+1}$ for all~$i$. 
	
	Consider the Magnus expansion of $F$ into the algebra of formal power series in noncommutative variables
$$\mathfrak{M}\colon F\hooklongrightarrow \Z[[X_1,\ldots,X_{n-1},M_1,\ldots,M_g, L_1,\ldots,L_g]]$$
which sends $x_i, m_j, l_j$ of $F$ to $1+X_i, 1+M_j, 1+L_j$ respectively. It is well known that $a\in F_k$ if and only if $\mathfrak{M}(a)- 1$ is a sum of monomials of degree at least $k$; see \cite[Section~5]{MKS}.
Denote by $P_k$ the set of homogeneous polynomials of degree $k$ in $\Z[[X_1,\ldots,X_{n-1},M_1,\ldots,M_g, L_1,\ldots,L_g]]$.
We have an isomorphism
$\mathfrak{M}_k\colon F_k/F_{k+1}\to P_k$ which sends $a F_{k+1}\in F_k/F_{k+1}$ to the homogeneous part of degree $k$ in $\mathfrak{M}(a)-1$ (see \cite[Corollary~5.12]{MKS}). Let $\mathfrak{M}_k(\alpha_i)=h_i$. The inverse image of $P_k \cap \Z[[X_1,\ldots, X_{n-1}]]$ under the above map is~$F'_k/F'_{k+1}$. Hence it is enough to show that $h_i$ is in $\Z[[X_1,\ldots, X_{n-1}]]$.

 By calculation, $\mathfrak{M}_{k+1}([x_i,a_i]) = X_i h_i- h_i X_i$, and then,
 $\mathfrak{M}_{k+1}(\prod_i [x_i,a_i])=\sum_i(X_i h_i- h_i X_i)=0$ since $\prod_i [x_i,\alpha_i] = 1$.
For a monomial $z$ in the $X_i,M_j,L_j$, define
\begin{displaymath}
l(z)=\left\{ \begin{array}{ll}
	t & \textrm{if $z=uyv$ where $v$ is a monomial in the $X_i$ of degree $t$,} \\
	& \hphantom{\textrm{if $z=uyv$ where }} \textrm{$y\in\{M_1,\ldots,M_g,L_1,\ldots,L_g\}$,} \\ 
	& \hphantom{\textrm{if $z=uyv$ where }} \textrm{$u$ is a monomial in the $X_i,M_j,L_j$,} \\
	\infty & \textrm{if $z$ is a monomial in the $X_i$.}
	\end{array} \right.
\end{displaymath}
i.e. $l(z)$ is the length of the rightmost submonomial written in the $X_i$ only.
Write $h_i=\sum_\alpha c_{i\alpha} z_{i\alpha}$, where $c_{i\alpha}\in \Z$, $z_{i\alpha}$ is a monomial. To show that $l(z_{i\alpha})=\infty$ for all $i,\alpha$, suppose $l(z_{i\alpha})< \infty$ for some $i, \alpha$. Let $l:=\min_{i,\alpha}\{l(z_{i\alpha})\}$. Then $l$ is finite. Consider a term $c_{{i_0}{\alpha_0}} X_{i_0} z_{{i_0} {\alpha_0}}$ with $l(z_{{i_0}{\alpha_0}})=l$ in the expression
$$\sum_{i,\alpha} c_{i\alpha} X_i z_{i\alpha} - \sum_{i,\alpha} c_{i\alpha} z_{i\alpha} X_i =0.$$
The term $c_{{i_0}{\alpha_0}} X_{i_0} z_{{i_0}{\alpha_0}}$ is not eliminated with any other term of the form $c_{i\alpha} z_{i\alpha} X_i$ since $l(X_{i_0} z_{{i_0}{\alpha_0}})<l(z_{i\alpha} X_i)$. Also, $c_{{i_0}{\alpha_0}}X_{i_0} z_{{i_0}{\alpha_0}}$ is not eliminated with any other term of the form $c_{i\alpha} X_i z_{i\alpha}$ if either $i\neq {i_0}$ or $l(z_{i\alpha})\neq l$.
It follows that $\sum_{\{{\alpha} | l(z_{{i_0}{\alpha}})=l\}} c_{{i_0}{\alpha}} X_{i_0} z_{{i_0}{\alpha}} = 0$, and so $\sum_{\{{\alpha} | l(z_{{i_0}{\alpha}})=l\}} c_{{i_0}{\alpha}} z_{{i_0}{\alpha}} = 0$. Thus, all the terms in $h_i$ with $l(-)=l$ are eliminated in~$h_i$. This contradicts the choice of~$l$.
\end{proof}

\begin{lemma} \phantomsection
\label{lemma:keys} 	
\begin{enumerate}	
		\item For each $i$, the map $F_k/F_{k+1} \to F_{k+1}/F_{k+2}$ which sends $\alpha$ to $[x_i,\alpha]$ is an injective homomorphism if $k\geq2$. If $k=1$, then the restriction map $H_1(\langle x_{i'}, m_j, l_j~|~i'\neq i \rangle) \to F_2/F_3$ is injective.
		\item Let $\phi\in\A_{k+2,k}$. Then $\phi(x_i)=(\alpha_i^\phi)^{-1} x_i \alpha_i^\phi$, $\phi_{k+1}(m_j)=m_j \beta_j^\phi$, $\phi_{k+1}(l_j)=l_j \gamma_j^\phi$ for some $\alpha_i^\phi ,\beta_j^\phi,\gamma_j^\phi \in F_k/F_{k+1}$. The formula $\phi(\prod x_i \prod [m_j,l_j])=\prod x_i \prod [m_j,l_j]$ implies 
		$$\prod [x_i,\alpha_i^\phi] \prod [m_j,\gamma_j^\phi][\beta_j^\phi,l_j][\beta_j^\phi,\gamma_j^\phi]=1 \textrm{ in }F/F_{k+2}. $$
			If $k\geq 2$, $[\beta_j^\phi,\gamma_j^\phi]=1$ in $F/F_{k+2}$.
		\item For a group $G$, if an automorphism of $G/G_k$ induces the identity on $G/G_l$ for some $l<k$, then it is also identity on $G_{1+\alpha}/G_{l+\alpha}$ for any $\alpha \leq k-l$.	\end{enumerate}
\end{lemma}

\begin{proof}
	\begin{enumerate}
		\item We use the Magnus expansion as in the proof of Lemma~\ref{lemma:D'}. For $\alpha\in F_k/F_{k+1}$, let $\mathfrak{M}_k(\alpha)=h$. Then $\mathfrak{M}_{k+1}([x_i,\alpha])=X_i h- h X_i$. Suppose $[x_i,\alpha]=1$ in $F_{k+1}/F_{k+2}$, then $X_i h- h X_i=0$. A similar argument to the proof gives that $h$ is a polynomial in only $X_i$ with the given~$i$. It implies $\alpha\in \langle x_i \rangle$. If $k\geq 2$, $\alpha=0$. If $k=1$, $\alpha= x_i^t$ for some $t\in\Z$.
		\item We omit the calculation.
		\item Use induction on $\alpha$. From the hypothesis, it is true for $\alpha=0$. Assume that $\phi_{l+\alpha-1}|_{G_\alpha/G_{l+\alpha-1}} = \id$. It is enough to check that $\phi_{l+\alpha}([g,h])=[g,h]$ for $g\in G, h \in G_\alpha$. Let $\phi_{l+\alpha}(g)=ga$ with $a\in G/G_{l+\alpha}$. Then $a$ is in $G_l/G_{l+\alpha}$. Also, $\phi_{l+\alpha}(h)=hb$ for some $b \in G_{l+\alpha-1}/G_{l+\alpha}$. Applying commutators identityes, we see that $\phi_{l+\alpha}([g,h])=[ga,hb]=[ga,b][ga,h][[ga,h],b]=[ga,h]=[g,h][[g,h],a][a,h]=[g,h]$ in~$G/G_{l+\alpha}$.
	\end{enumerate}
\end{proof}
 
\begin{remark}
	For $M\in\H^0$, the boundary condition $\tilde\eta_2(M)([\partial_n])=[\partial_n]$ implies $\prod_j [m_j,\mu''_2(M)_j][\mu'_2(M)_j, \mu''_2(M)_j][\mu'_2(M)_j,l_j]=1 \in F_2/F_3$. The term $[\mu'_k(M)_j, \mu''_k(M)_j]$ does not vanishes in $F_k/F_{k+1}$ for the case $k=2$, differently from the other~$k$.
It is the cause of failure for $(\mu'_2,\mu''_2)$ to be a homomorphism on $\H(1)/\H^0[2]$ and that for $\H(1)/\H(2) \to \D_1(H)$ to be defined when $g\neq 0$.
\end{remark}

\begin{proof}(Proof of Theorem~\ref{theorem:Milnor} and Theorem~\ref{theorem:rank})
	We will prove that there is a surjection $\Theta \colon \A_{k+2,k}\to \Ker\p_k$ such that $\Theta \circ \tilde\eta_{k+2} =\tilde\mu_{k+1}$ on~$\H(k)$.
	
	Suppose  $k\geq 2$ and $\phi\in\A_{k+2,k}$.
	The $\beta_j^\phi,\gamma_j^\phi \in F_k/F_{k+1}$ in Lemma~\ref{lemma:keys}(2) are uniquely determined. If we consider $x_i^{-1} \phi(x_i)=[x_i^{-1}, (\alpha_i^{\phi})^{-1}]=[x_i, \alpha_i^\phi]\in F_{k+1}/F_{k+2}$, Lemma~\ref{lemma:keys}(1) implies that $\alpha_i^\phi$ is uniquely determined in $F_k/F_{k+1}$ for $k\geq 2$. By Lemma~\ref{lemma:keys}(2), $(\alpha_1^\phi,\ldots,\alpha_{n-1}^\phi,\beta_1^\phi,\ldots,\beta_g^\phi, \gamma_1^\phi,\ldots,\gamma_g^\phi)$ is in $\Ker\p_k$.
	Hence, we can define a map
	\begin{align*}
		\A_{k+2,k} &\xrightarrow{\Theta} \Ker\p_k \subset(F_k/F_{k+1})^{2g+n-1}\\
		\phi			&\longmapsto					(\alpha_1^\phi,\ldots,\alpha_{n-1}^\phi,\beta_1^\phi,\ldots,\beta_g^\phi, \gamma_1^\phi,\ldots,\gamma_g^\phi). 
	\end{align*}
The relation (\ref{eqn:relation}) in Section~\ref{subsec:Milnor} implies $\Theta \circ \tilde\eta_{k+2} =\tilde\mu_{k+1}$ on~$\H(k)$.
	
	To see the surjectivity of $\Theta$, let $\theta = (\alpha_1,\ldots,\alpha_{n-1},\beta_1,\ldots,\beta_g, \gamma_1,\ldots,\gamma_g)\in\Ker\p_k$ with $\alpha_i, \beta_j, \gamma_j \in F_k/F_{k+1}$.
	We show that there are lifts $\tilde\alpha_i, \tilde\beta_j, \tilde\gamma_j \in F_k/F_{k+2}$ of $\alpha_i,\beta_j,\gamma_j$ such that the automorphism $\phi$ of $F/F_{k+2}$ sending $x_i,m_j,l_j$ to $\alpha_i^{-1} x_i \alpha_i, m_j \tilde\beta_j,l_j \tilde\gamma_j$ respectively is in $\Aut_*(F/F_{k+2})$ with a lift $\tilde\phi\in\Aut(F/F_{k+3})$ sending $x_i$ to $\tilde\alpha_i^{-1} x_i \tilde\alpha_i$.
	In detail, let $\tilde\alpha_i = \alpha_i a_i$, $\tilde\beta_j=\beta_j b_j$, $\tilde\gamma_j= \gamma_j c_j$ with $a_i, b_j, c_j \in F_{k+1}/F_{k+2}$.
 	 Then $\tilde\phi(x_i)=a_i^{-1} \phi(x_i) a_i$, $\phi(m_j)=\phi_{k+1}(m_j)b_j$, $\phi(l_j)=\phi_{k+1}(l_j)c_j$. The $\tilde\phi([\partial_n])$ is computed as
	\begin{align*}
		\tilde\phi&\bigg(\prod x_i \prod [m_j,l_j]\bigg)\\
			&=\prod a_i^{-1} \phi(x_i) a_i \prod [\phi_{k+1}(m_j)b_j, \phi_{k+1}(l_j)c_j]	\\
			&=\prod \phi(x_i) [\phi(x_i)^{-1},a_i^{-1}] \prod [\phi_{k+1}(m_j),\phi_{k+1}(l_j)] [\phi_{k+1}(m_j), c_j] [b_j, \phi_{k+1}(l_j)] [b_j, c_j] \\
			&=\prod \phi(x_i) [\alpha_i^{-1} x_i^{-1}\alpha_i ,a_i^{-1}]  \prod [\phi_{k+1}(m_j),\phi_{k+1}(l_j)] [m_j \beta_j, c_j] [b_j, l_j \gamma_j] [b_j, c_j] \\
			&=\prod_i \phi(x_i) \prod_j [\phi_{k+1}(m_j),\phi_{k+1}(l_j)] \prod_i [x_i^{-1}, a_i^{-1}] \prod_j [m_j,c_j][\beta_j,c_j][b_j, l_j][b_j,\gamma_j][b_j,c_j] \\
			&=\prod_i \phi(x_i) \prod_j [\phi_{k+1}(m_j),\phi_{k+1}(l_j)] \prod_i [x_i, a_i] \prod_j [m_j,c_j][b_j, l_j] \quad \textrm{ if } k\geq 2\\
			&=\prod_i x_i \prod_j [m_j,l_j] \cdot \xi \cdot \prod_i [x_i, a_i] \prod_j [m_j,c_j][b_j, l_j] \quad \textrm{ for some } \xi \in F_{k+2}/F_{k+3}. 
	\end{align*} 
	Since any element in $F_{k+2}/F_{k+3}$ can be represented by $\prod [x_i, a_i] \prod [m_j,c_j][b_j,l_j]$ for $a_i, b_j, c_j \in F_{k+1}/F_{k+2}$, we can choose $a_i, b_j, c_j$ so that $\xi^{-1}=\prod [x_i,a_i] \prod [m_j,c_j][b_j,l_j]$.
		Thus $\tilde\phi$ fixes $[\partial_n]$, and hence $\phi\in \Aut_*(F/F_{k+2})$.
		The elements $\alpha_i,\beta_j,\gamma_j \equiv 1$ mod $F_k$, so $\phi_{k+1}(x_i)=x_i, \phi_k(m_j)=m_j, \phi_k(l_j)=l_j$, and $\phi \in\A_{k+2,k}$.
		Therefore, $\Theta(\phi)=\theta$.
		
		Lastly, we show that $\Theta$ is a homomorphism. For $\phi, \psi \in \A_{k+2,k}$,
		$$\Theta(\phi)+\Theta(\psi)=(\alpha_1^\phi \alpha_1^\psi,\ldots,\alpha_{n-1}^\phi \alpha_{n-1}^\psi,\beta_1^\phi\beta_1^\psi,\ldots,\beta_g^\phi\beta_g^\psi,\gamma_1^\phi\gamma_1^\psi,\ldots,\gamma_g^\phi\gamma_g^\psi).$$
		
		By Lemma~\ref{lemma:keys}~(3), $\psi_{k+1}|_{F_2/F_{k+1}}=\id$ follows from $\psi_k=\id$. Hence,
		\begin{align*}
			(\psi\circ\phi)(x_i)&=\psi((\alpha_i^\phi)^{-1} x_i \alpha_i^\phi)\\
									&=\psi_{k+1}((\alpha_i^\phi)^{-1}) (\alpha_i^\psi)^{-1} x_i \alpha_i^\psi \psi_{k+1}(\alpha_i^\phi)	\\
									&= (\alpha_i^\phi)^{-1}(\alpha_i^\psi)^{-1} x_i \alpha_i^\psi \alpha_i^\phi
		\end{align*}
		We obtain $\alpha_i^{\psi\circ\phi} = \alpha_i^\psi \alpha_i^\phi$. Also,
		\begin{align*}
			(\psi\circ\phi)_{k+1}(m_j)&=\psi_{k+1}(m_j\beta_j^\phi) \\
									&=m_j \beta_j^\psi \psi_{k+1}(\beta_j^\phi) \\
									&=m_j \beta_j^\psi \beta_j^\phi
		\end{align*}
		since $\psi_{k+1}|_{F_2/F_{k+1}}=\id$. So, $\beta_i^{\psi\circ\phi} = \beta_i^\psi \beta_i^\phi$.
		Similarly for $l_j$.
		Since $F_k/F_{k+1}$ is abelian, $\Theta(\phi)+\Theta(\psi)=\Theta(\psi\circ\phi)$.
		
		To show the surjectivity of $\H^0[k+1]\to\Ker\p'_k$ for $k\geq 2$, we observe~$\Theta(\K_{k+2,k+1})$.
For $\phi \in \K_{k+2,k+1}$, $\alpha_i^\phi$ is uniquely determined for each $i$ if $k\geq 2$ and $\beta_i^\phi$, $\gamma_i^\phi$ are trivial in $F_k/F_{k+1}$, hence the $\Theta(\phi)$ is in $(F_k/F_{k+1})^{n-1}\times 0^{2g} \subset (F_k/F_{k+1})^{2g+n-1}$. By Lemma~\ref{lemma:D'}, $\Theta(\K_{k+2,k+1})$ can be considered as a subset of $\Ker\p'_k$.
Let $\theta\in \Ker\p'_k$ be $(\alpha_1,\ldots,\alpha_{n-1})$ satisfying $\prod [x_i,\alpha_i]=1$. It is the case that $\beta_j=\gamma_j=1$ for all $j$ in the proof of the surjectivity of $\Theta$. Similar to the argument, we obtain $\phi \in \A_{k+2,k}$ such that $\phi(x_i)=(\alpha_i a_i)^{-1} x_i \alpha_i a_i$, $\phi(m_j)=m_j b_j$, $\phi(l_j)=l_j c_j$ for some $a_i, b_j, c_j \in F_{k+1}/F_{k+2}$ and $\Theta(\phi)=\theta$. The $\phi$ is in $\K_{k+2,k+1}$.
	
		To prove the last isomorphism $\frac{\H[2]}{\H(2)} \cong \Ker\p'_k$ in Theorem~\ref{theorem:rank}, we construct an isomorphism $$\bar\Theta^S\colon \Z^{n-1} \times \frac{\K_{3,2}}{\A_{3,2}} \to \Ker \{\p_k|\colon (F/F_2)^{n-1}\times 0^{2g} \to F_2/F_3\}.$$
For $\phi \in \K_{3,2}$, $\alpha_i^\phi$ is uniquely determined in $H_1(\langle x_{i'}, m_j, l_j~|~i'\neq i \rangle)$ by Lemma~\ref{lemma:keys}(1). For an element $((t_i), \phi) \in \Z^{n-1} \times \frac{\K_{3,2}}{\A_{3,2}}$, $\bar\Theta^S$ maps it to $(x_1^{t_1}\alpha_1^\phi,\ldots,x_{n-1}^{t_{n-1}}\alpha_{n-1}^\phi,1,\ldots,1)$. Then we can check similarly that the map is an isomorphism. 
\end{proof}

\end{document}